\documentclass[12pt]{article}
\usepackage{epsf}
\usepackage{amsbsy,amsmath}
\usepackage{mathtools}
\usepackage{mathrsfs}
\usepackage{amsfonts}
\usepackage{amssymb}
\usepackage{enumitem}
\usepackage{eucal}
\usepackage{graphics,mathrsfs}
\usepackage{amsthm}
\usepackage{secdot}

\newtheorem{theorem}{Theorem}[section]
\newtheorem{lemma}[theorem]{Lemma}
\newtheorem{proposition}[theorem]{Proposition}
\newtheorem{corollary}[theorem]{Corollary}

\theoremstyle{definition}
\newtheorem{definition}[theorem]{Definition}

\theoremstyle{remark}
\newtheorem{remark}[theorem]{Remark}
\numberwithin{equation}{section}

\numberwithin{equation}{section}

\newsavebox{\savepar}

\begin{document}
	
	\title{Multiplicity and H\"{o}lder regularity of solutions for a nonlocal elliptic PDE involving singularity}
	\author{Kamel Saoudi, Sekhar Ghosh \&  Debajyoti Choudhuri\footnote{Corresponding
			author: dc.iit12@gmail.com} \\
		\small{Department of Mathematics, National Institute of Technology Rourkela, India}\\
		\small{Department of Mathematics, Imam Abdul Rahman Bin Faisal University,}\\ \small{Dammam, Saudi Arabia}\\
		\small{Emails: sekharghosh1234@gmail.com \& dc.iit12@gmail.com}}
	\date{}
	\maketitle
	
	\begin{abstract}
		\noindent In this paper we prove the existence of multiple solutions for a nonlinear nonlocal elliptic PDE involving a singularity which is given as
		\begin{eqnarray}
		(-\Delta_p)^s u&=& \frac{\lambda}{u^\gamma}+u^q~\text{in}~\Omega,\nonumber\\
		u&=&0~\text{in}~\mathbb{R}^N\setminus\Omega,\nonumber\\
		u&>& 0~\text{in}~\Omega\nonumber,
		\end{eqnarray}
		where $\Omega$ is an open bounded domain in $\mathbb{R}^N$ with smooth boundary, $N>ps$, $s\in (0,1)$, $\lambda>0$, $0<\gamma<1$, $1<p<\infty$, $p-1<q\leq p_s^{*}=\frac{Np}{N-ps}$. We employ variational techniques to show the existence of multiple positive weak solutions of the above problem. We also prove that for some $\beta\in (0,1)$, the weak solution to the problem is in $C^{1,\beta}(\overline{\Omega})$.\\
		{\bf keywords}: Elliptic PDE, PS condition, Mountain Pass theorem, G\^{a}teaux derivative.
	\end{abstract}
	\section{Introduction}
	\noindent The study of elliptic PDEs involving fractional $p$-Laplacian operator plays an important role in many field of sciences, like in the field of finance, optimization, electromagnetism, astronomy, water waves, fluid dynamics, probability theory, phase transitions etc. The application to L\'{e}vy processes in probability theory can be seen in \cite{bertoin1996levy, bojdecki1999fractional} and that in finance can be seen in \cite{cont2004financial}. For further details on applications, one can refer to \cite{valdinoci2009long} and the references therein.\\
	\noindent In the recent past, a vast investigation has been done for the following local problem.
	\begin{eqnarray}\label{refer local}
	-\Delta_p u&=& \frac{\lambda a(x)}{u^\gamma}+Mu^q~\text{in}~\Omega,\nonumber\\
	u&=&0~\text{in}~\partial\Omega,\\
	u&>& 0~\text{in}~\Omega\nonumber,
	\end{eqnarray}
	where $1<p<N$, $M\geq0$, $a:\Omega\rightarrow\mathbb{R}$ is a nonnegative bounded function.
	For $M=0$, the existence of weak solutions and regularity of solutions for singular problem as in (\ref{refer local}) has been widely studied in \cite{boccardo2010semilinear, crandall1977dirichlet, diaz1987elliptic, lazer1991singular, rosen1971minimum} and the references therein.
	Recently, for $M\neq0$ the problem (\ref{refer local}) has been studied to show the existence of multiple solutions in \cite{adimurthi2006multiplicity, ambrosetti1994combined, bal2017multiplicity, coclite1989singular, crandall1977dirichlet, dhanya2012global, giacomoni2009multiplicity, giacomoni2007sobolev, giacomoni2007multiplicity, haitao2003multiplicity, hirano2004existence} and the references therein. In most of these references, the multiplicity results were proved by the variational methods, Nehari manifold method, perron method, concentration compactness and the moving hyperplane method.\\ 
	For $p=2$, $a(x)\equiv 1$ and $M=1$, Haitao \cite{haitao2003multiplicity} used the variational method to show that for $0<\lambda<\Lambda<\infty$, the problem (\ref{refer local}) has two solutions. For $a(x)\equiv 1$ and $M=1$ Giacomoni \& Sreenadh \cite{giacomoni2007multiplicity} had studied the problem (\ref{refer local}) for $1<p<\infty$, to show the existence of atleast two solutions by using shooting method. Later Giacomoni et al. \cite{giacomoni2007sobolev} had proved the multiplicity result using the variational method. In \cite{bal2017multiplicity} the authors showed the multiplicity of solutions using the moving hyperplane method. In \cite{dhanya2012global} the authors applied the concentration compactness method to establish the multiplicity results. The Nehari manifold method is used in \cite{yijing2001combined} to show the existence of multiple solutions.\\
	\noindent Recently, the study of nonlocal elliptic PDEs involving singularity with Dirichlet boundary condition has drawn interest by many researchers. The investigation for the existence of weak solutions for a nonlocal elliptic pdes with concave-convex type nonlinearity, i.e. $u^p+\lambda u^q$, for $p, q>0$ has been extensively studied in \cite{barrios2015critical, bisci2015brezis, bisci2015lower, brandle2013concave, choudhuri2015existence, wei2015multiplicity} and the references therein. The existence results for the  Brezis-Nirenberg type problem has been studied in \cite{bisci2015brezis, mosconi2016brezis, servadei2013brezis}. The eigenvalue problem for fractional $p$-Laplacian and the properties of first and second eigen values can be found in \cite{brasco2016second, franzina2013fractional, lindgren2014fractional}.\\
	\noindent The following nonlocal problem has been studied by many authors,
	\begin{eqnarray}\label{refer nonlocal}
	(-\Delta_p)^s u&=& \frac{\lambda a(x)}{u^\gamma}+Mf(x,u)~\text{in}~\Omega,\nonumber\\
	u&=&0~\text{in}~\mathbb{R}^N\setminus\Omega,\\
	u&>& 0~\text{in}~\Omega\nonumber,
	\end{eqnarray}
	where $N>ps$, $M\geq0$, $a:\Omega\rightarrow\mathbb{R}$ is a nonnegative bounded function.
	For $p=2$, $M=0$, $\lambda=1$ and $a(x)\equiv1$ the problem (\ref{refer nonlocal}) was studied by  \cite{fang2014existence}. In \cite{fang2014existence} the author had shown that for $0<\gamma<1$, the problem \eqref{refer nonlocal} has a unique solution $u\in C^{2,\alpha}(\Omega)$ for $0<\alpha<1$. For $1<p<\infty$, $M=0$ and $\lambda=1$ the problem (\ref{refer nonlocal}) was studied by Canino et al. \cite{canino2017nonlocal}. For $0<\gamma<1$ and $\lambda=1$, Ghanmi \& Saoudi \cite{ghanmi2016nehari} established the existence of two solutions by Nehari manifold method for fractional Laplacian operator. For $p=2$, $f(x,u)=u^{2_s^*-1}$, Mukherjee \& Sreenadh \cite{mukherjee2016fractional} studied the problem (\ref{refer nonlocal}) by variational method. Ghanmi \& Saoudi \cite{ghanmi2016multiplicity} guaranteed the existence of multiple weak solutions to the problem (\ref{refer nonlocal}), for $0<\gamma<1$ and $1<p-1<q\leq p_s^*$. The authors in \cite{ghanmi2016multiplicity} have used the Nehari manifold method.\\
	Besides the existence and multiplicity of solutions, the regularity of these solutions for PDEs involving a nonlocal elliptic operator with a power nonlinearity has been studied by Iannizzotto et al \cite{iannizzotto2015hs}, K. Saoudi \cite{saoudi2017vs}. For the results on the local operator with a power nonlinearity refer to \cite{brezis1993h1, saoudi2014p(x)} and the references therein. The regularity of solution to the PDE involving a local elliptic operator, viz.  Laplacian, p-Laplacian, involving singularity and a power nonlinearity has been studied by \cite{giacomoni2010w01} and \cite{giacomoni2007sobolev} respectively. However, to our knowledge there are no investigation on regularity of solutions to a PDE invovling a nonlocal elliptic operator with a singularity and a power nonlinearity.
	\noindent In the present article we will prove the existence of multiple solutions and its regularity for the following nonlocal problem
	\begin{eqnarray}\label{main}
	(-\Delta_p)^s u&=& \frac{\lambda}{u^\gamma}+u^q~\text{in}~\Omega,\nonumber\\
	u&=&0~\text{in}~\mathbb{R}^N\setminus\Omega,\\
	u&>& 0~\text{in}~\Omega\nonumber,
	\end{eqnarray}
	where $\Omega$ is a open bounded domain of $\mathbb{R}^N$ with smooth boundary, $N>ps$, $s\in (0,1)$, $\lambda>0$, $0<\gamma<1$, $1<p<\infty$, $p-1<q\leq p_s^{*}=\frac{Np}{N-ps}$ and $(-\Delta_p)^s$ is the fractional $p$-Laplacian operator which is defined as
	$$(-\Delta_p)^s u(x)=C_{N, s} P.V. \int_{\mathbb{R}^N}\cfrac{|u(x)-u(y)|^{p-2}(u(x)-u(y))}{|x-y|^{N+ps}}dy~,\forall~p\in[1,\infty)$$ with $C_{N, s},$ being the normalizing constant.\\
	\noindent Similar problems to that in (\ref{main}) has been studied by a few authors like Mukherjee \& Sreenadh \cite{mukherjee2016dirichlet}, Saoudi \cite{saoudi2017critical}. In \cite{mukherjee2016dirichlet}, the authors established the existence of multiple solutions by using the Nehari manifold method. In \cite{saoudi2017critical}, for $p=2$ the multiplicity result for the problem (\ref{main}) is proved  with the help of the variational method, where the author proved the existence result by converting the nonlocal problem to a local problem.\\
	\noindent In this article we show the existence of multiple solutions to the nonlocal problem (\ref{main}) by combining some variational techniques developed in \cite{ambrosetti1973dual}. We first show the existence of a weak solution using sub-super solution method. To show the existence of a second solution, we use a modified version of the Mountain Pass lemma due to Ambrosetti \& Rabinowitz \cite{ambrosetti1973dual}, which can be found in Ghoussoub \& Preiss \cite{ghoussoub1989general}. We further proved that the weak solution to (\ref{main}) is in $C^{1,\beta}(\overline{\Omega})$ for some $\beta\in (0,1)$.\\
	\noindent The article is organised in the following sequence. 
	In Section 2 we give the mathematical formulation with the appropriate functional analytic setup. Section 3 is devoted to establish the existence of a weak solution. In Section 4 we prove the multiplicity of solutions using the Ekeland's variational principle.\\
	The main results proved in this manuscript are the followings\\
	\begin{theorem}\label{main thm}
		There exists $\Lambda\in(0,\infty)$ such that,
		\begin{itemize}
			\item[(i)] $\forall ~\lambda\in(0, \Lambda)$, the problem (\ref{main}) has a minimal solution.
			\item[(ii)] For $\lambda=\Lambda$ the problem (\ref{main}) has atleast one solution.
			\item[(iii)] $\forall ~\lambda\in(\Lambda, \infty)$ the problem (\ref{main}) has no solution.
		\end{itemize}
	\end{theorem}
	
	\begin{theorem}\label{multiplicity thm}
		For every $\lambda\in(0, \Lambda)$, the problem (\ref{main}) has multiple solutions.
	\end{theorem}
	\begin{theorem}
		With the growth conditions of f in tact.  Let $u_0\in C^1(\overline{\Omega})$satisfying 
		\begin{equation}\label{bord}
		u_0\geq \eta\mbox{d}(x,\partial\Omega)^s\mbox{ for some }\eta>0
		\end{equation}
		be a local minimizer of $I$ in $C^1(\overline{\Omega})$ topology; that is,
		\begin{eqnarray*}
			\exists\,\epsilon >0\,\mbox{such that }\,u\in C^1(\overline{\Omega})\;, 
			\|u-u_0\|_{C^1(\overline{\Omega})}<\epsilon\Rightarrow I(u_0)\leq I(u).
		\end{eqnarray*}
		Then, $u_0$ is a local minimum of $I$ in $W^{s,p}_0(\Omega)$ also.
	\end{theorem}
	
	\section{Mathematical formulation and Space setup}
	This section is entirely devoted to a brief discussion about a few definitions,  notations and function spaces which will be used henceforth in this manuscript. We begin by defining the following function space.  Let $\Omega\subset\mathbb{R}^N$ and $Q=\mathbb{R}^{2N}\setminus((\mathbb{R}^N\setminus\Omega)\times(\mathbb{R}^N\setminus\Omega))$, then the space $(X, \|\|_X)$ is defined by
	\begin{eqnarray}
	X&=&\left\{u:\mathbb{R^N}\rightarrow\mathbb{R}~\text{is measurable}, u|_{\Omega}\in L^p(\Omega) ~\text{and}~\frac{|u(x)-u(y)|}{|x-y|^{\frac{N+ps}{p}}}\in L^{p}(Q)\right\}\nonumber
	\end{eqnarray}
	equipped with the Gagliardo norm 
	\begin{eqnarray}
	\|u\|_X&=&\|u\|_{p}+\left(\int_{Q}\frac{|u(x)-u(y)|^p}{|x-y|^{N+ps}}dxdy\right)^{\frac{1}{p}}.\nonumber
	\end{eqnarray}
	Here $\|u\|_{p}$ refers to the $L^p$-norm of $u$. We further define the space 
	\begin{eqnarray}
	X_0&=&\left\{u\in X: u=0 ~\text{a.e. in}~ \mathbb{R}^N\setminus\Omega\right\}\nonumber
	\end{eqnarray}
	equiped with the norm
	\begin{eqnarray}
	\|u\|&=&\left(\int_{Q}\frac{|u(x)-u(y)|^p}{|x-y|^{N+ps}}dxdy\right)^{\frac{1}{p}}.\nonumber
	\end{eqnarray}
	\noindent The best Sobolev constant is defined as 
	\begin{equation}\label{sobolev const}
	S=\underset{u\in X_0\setminus\{0\}}{\inf}\cfrac{\int_{Q}\frac{|u(x)-u(y)|^p}{|x-y|^{N+ps}}dxdy}{\left(\int_\Omega|p_s^*|dx\right)^{\frac{p}{p_s^*}}}
	\end{equation}
	We now define a weak solution to the problem defined in (\ref{main}).
	\begin{definition}\label{defn weak}
		A function $u\in X_0$ is a weak solution to the problem (\ref{main}), if 
		\begin{align*}
		&(i)~ u>0, ~u^{-\gamma}\phi\in L^1(\Omega) ~\text{and}\\
		&(ii)~ \int_{Q}\frac{|u(x)-u(y)|^{p-2}(u(x)-u(y))(\phi(x)-\phi(y))}{|x-y|^{N+ps}}dxdy-\int_{\Omega}\frac{\lambda}{u^{\gamma}}\phi-\int_{\Omega}u^q\phi=0
		\end{align*}
		for each $\phi\in X_0.$
	\end{definition}
	\noindent Following are the definitions of a sub and a super solution to the problem (\ref{main}).
	\begin{definition}\label{defn subsoln}
		A function $\underline{u}\in X_0$ is a weak subsolution to the problem (\ref{main}), if 
		\begin{align*}
		&(i)~ \underline{u}>0, ~\underline{u}^{-\gamma}\phi\in L^1(\Omega) ~\text{and}\\
		&(ii)~ \int_{Q}\frac{|\underline{u}(x)-\underline{u}(y)|^{p-2}(\underline{u}(x)-\underline{u}(y))(\phi(x)-\phi(y))}{|x-y|^{N+ps}}dxdy-\int_{\Omega}\frac{\lambda}{\underline{u}^{\gamma}}\phi-\int_{\Omega}\underline{u}^q\phi\leq0
		\end{align*}
		for every nonnegative $\phi\in X_0.$
	\end{definition}
	\begin{definition}\label{defn supsoln}
		A function $\bar{u}\in X_0$ is a weak supersolution to the problem (\ref{main}), if 
		\begin{align*}
		&(i)~ \bar{u}>0, ~\bar{u}^{-\gamma}\phi\in L^1(\Omega) ~\text{and}\\
		&(ii)~ \int_{Q}\frac{|\bar{u}(x)-\bar{u}(y)|^{p-2}(\bar{u}(x)-\bar{u}(y))(\phi(x)-\phi(y))}{|x-y|^{N+ps}}dxdy-\int_{\Omega}\frac{\lambda}{\bar{u}^{\gamma}}\phi-\int_{\Omega}\bar{u}^q\phi\geq0
		\end{align*}
		for all nonnegative $\phi\in X_0.$
	\end{definition}
	\noindent We now list out the embedding results in the form of a Lemma pertaining to the function space $X_0$ \cite{servadei2012mountain, servadei2013variational}. 
	\begin{lemma}\label{embedding}The following embedding results holds for the space $X_0$.
		\begin{enumerate}
			\item If $\Omega$ has a Lipshitz boundary, then the embedding $X_0 \hookrightarrow L^{q}(\Omega) $ for $q\in[1,p_s^*)$, where $p_s^*=\frac{Np}{N-ps}$.
			\item The embediing $X_0\hookrightarrow L^{p_s^*}(\Omega)$ is continuous. 
		\end{enumerate}
	\end{lemma}
	\noindent The main goal achieved in this article is the existence of two distinct, positive weak solutions to the problem \eqref{main} in $X_0$. To establish that we will engage ourselves to find the existence of two distinct critical points to the following energy functional, $I_\lambda\colon X_0\rightarrow\mathbb{R}$ defined as
	\begin{equation*}
	I_\lambda(u)=\frac{1}{p}\int_{Q}\frac{|u(x)-u(y)|^{p}}{|x-y|^{N+ps}}dxdy-\frac{\lambda}{1-\gamma}\int_{\Omega}(u^+)^{1-\gamma}dx-\frac{1}{q+1}\int_{\Omega}(u^+)^{q+1}dx.
	\end{equation*}
	Here, $u^+=\max\{u, 0\}$ and $u^-=\max\{-u, 0\}$. It is easy to observe that $I_\lambda$ is not $C^1$ due the presence of the singular term in it. Therefore the usual approach by Mountain Pass lemma \cite{ambrosetti1973dual} fails. So, we will proceed with a cut off functional argument as in Ghoussoub \& Preiss \cite{ghoussoub1989general}. 
	Let us define 
	\begin{equation*}
	\Lambda=\inf\{\lambda>0:~\text{The problem (\ref{main}) has no weak solution}\}
	\end{equation*}
	\section{Main results}
	We begin this section with the following two comparison results.
	\begin{lemma}[Weak Comparison Principle]\label{weak comparison}
		Let $u, v\in X_0$. Suppose, $(-\Delta_p)^sv-\frac{\lambda}{v^{\gamma}}\geq(-\Delta_p)^su-\frac{\lambda}{u^{\gamma}}$ weakly with $v=u=0$ in $\mathbb{R}^N\setminus\Omega$.
		Then $v\geq u$ in $\mathbb{R}^N.$
	\end{lemma}
	\begin{proof}
		Since, $(-\Delta_p)^sv-\frac{\lambda}{v^{\gamma}}\geq(-\Delta_p)^su-\frac{\lambda}{u^{\gamma}}$ weakly with $u=v=0$ in $\mathbb{R}^N\setminus\Omega$, we have
		\begin{align}\label{compprinci}
		\langle(-\Delta_p)^sv,\phi\rangle-\int_{\Omega}\frac{\lambda\phi}{v^{\gamma}}dx&\geq\langle(-\Delta_p)^su,\phi\rangle-\int_{\Omega}\frac{\lambda\phi}{u^{\gamma}}dx, ~\forall{\phi\geq 0\in X_0}.
		\end{align}
		In particular choose $\phi=(u-v)^{+}$. To this choice, the inequlity in \eqref{compprinci} looks as follows.
		\begin{align}\label{compprinci1}
		\langle(-\Delta_p)^sv-(-\Delta_p)^su,(u-v)^{+}\rangle-\int_{\Omega}\lambda(u-v)^{+}\left(\frac{1}{v^{\gamma}}-\frac{1}{u^{\gamma}}\right)dx&\geq 0.
		\end{align}
		Let $\psi=u-v$. 
		The identity 
		\begin{align}\label{identity}
		|b|^{p-2}b-|a|^{p-2}a&=(p-1)(b-a)\int_0^1|a+t(b-a)|^{p-2}dt
		\end{align}
		with $a=v(x)-v(y)$, $b=u(x)-u(y)$ gives
		\begin{align}
		|u(x)-u(y)|^{p-2}(u(x)-u(y))-|v(x)-v(y)|^{p-2}(u(x)-u(y))\nonumber\\
		=(p-1)\{(u(y)-v(y))-(u(x)-v(x))\}Q(x,y)
		\end{align}
		where 
		\begin{align}
		Q(x,y)&=\int_0^1|(u(x)-u(y))+t((v(x)-v(y))-(u(x)-u(y)))|^{p-2}dt.
		\end{align}
		We choose the test function $\phi=(u-v)^{+}$. We express,
		\begin{align}
		\psi&=u-v\nonumber\\
		&=(u-v)^{+}-(u-v)^{-}\nonumber
		\end{align}
		to further obtain
		\begin{align}\label{negativity}
		[\psi(y)-\psi(x)][\phi(x)-\phi(y)]&=-(\psi^{+}(x)-\psi^{+}(y))^2.
		\end{align}
		The equation in \eqref{negativity} implies
		\begin{align}
		0&\geq \langle(-\Delta_p)^sv-(-\Delta_p)^su,(v-u)^{+}\rangle\nonumber\\&=-(p-1)\frac{Q(x,y)}{|x-y|^{N+sp}}(\psi^{+}(x)-\psi^{+}(y))^2\nonumber\\&\geq0.
		\end{align}
		This leads to the conclusion that the Lebesgue measure of $\Omega^{+}$, i.e., $|\Omega^{+}|=0$. In other words $v\geq u$ a.e. in $\Omega$.
	\end{proof}
	\begin{lemma}\label{strong}
		Consider the problem 
		\begin{align}\label{app1}
		(-\Delta_p)^su&=f~\text{in}~\Omega\nonumber\\
		u&=g~\text{in}~\mathbb{R}^N\setminus\Omega,
		\end{align}
		where $f\in L^{\infty}(\Omega)$ and $g\in W^{s,p}(\mathbb{R}^N)$. If $u\in W^{s,p}(\mathbb{R}^N)$ is a weak super-solution of \eqref{app1} with $f=0$ and $g\geq 0$, then $u\geq 0$ a.e. and admits a lower semicontinuous representation in $\Omega$.
	\end{lemma}
	\begin{proof}
		We first show that $u\geq 0 \in W^{s,p}(\mathbb{R}^N)$. We already have $u=g\geq 0$ in $\mathbb{R}^N\setminus\Omega$. Thus by an elementary inequality $(a-b)(a_{-}-b_{-})\leq -(a_{-}-b_{-})^2$ we have
		\begin{align}\label{uminusiszero}
		0&\leq \int_{\mathbb{R}^{2N}}\frac{|u(x)-u(y)|^{p-2}(u(x)-u(y))(u_{-}(x)-u_{-}(y))}{|x-y|^{N+ps}}dxdy\nonumber\\
		&\leq -\int_{\mathbb{R}^{2N}}\frac{|u(x)-u(y)|^{p-2}(u_{-}(x)-u_{-}(y))^2}{|x-y|^{N+ps}}dxdy.
		\end{align}
		From \eqref{uminusiszero} we have $u\geq 0$ a.e. in $\Omega$.\\
		We now define the following 
		\begin{align}
		u^{*}(x_0)&=\operatorname{ess}\liminf_{x\rightarrow x_0}u(x).
		\end{align}
		Cleary $u^{*}\geq 0$ a.e. in $\Omega$ since $u\geq 0$ a.e. in $\Omega$. Let $x_0\in \Omega$ be a Lebesgue point. By the definition of a Lebesgue point we have
		\begin{align}\label{5.11}
		u(x_0)&=\lim_{r\rightarrow 0^{+}}\def\avint{\mathop{\,\rlap{-}\!\!\int}\nolimits}\avint_{B_r(x_0)}u(x)dx\nonumber\\
		&\geq \lim_{r\rightarrow 0^{+}}\operatorname{ess}\inf_{B_r(x_0)}u(x)\nonumber\\
		&=u^{*}(x_0).
		\end{align}
		We now prove the reverse inequality, i.e. $u\leq u^{*}$ a.e. in $\Omega$.  Consider the function $-u$, which serves as a weak subsolution to \eqref{app1}, with $k=u(x_0)$ to obtain 
		\begin{align}
		\operatorname{ess}\sup_{B_{r/2}(x_0)}(-u)&\leq -u(x_0)+\text{Tail}((u(x_0)-u)_{+};x_0,r/2)\nonumber\\
		&+ C\left(\def\avint{\mathop{\,\rlap{-}\!\!\int}\nolimits}\avint_{B_r(x_0)}(u(x_0)-u(x))_{+}^p \,dx\right)^{1/p}
		\end{align}
		Passing the limit $r\rightarrow 0^{+}$, we have
		\begin{align}
		\lim_{r\rightarrow 0^{+}}\left(\def\avint{\mathop{\,\rlap{-}\!\!\int}\nolimits}\avint_{B_r(x_0)}(u(x_0)-u(x))_{+}^p dx\right)^{1/p}&=0.
		\end{align}
		since, $x_0\in\Omega$ is a Lebesgue point. Also by the H\"{o}lder's inequality we have 
		\begin{align}
		\text{Tail}((u(x_0)-u)_{+};x_0,r/2)&\leq r^{2s}\left(\int_{\mathbb{R}^{2N}\setminus B_r(x_0)}\frac{(u(x_0)-u(x))_{+}^{p}}{|x_0-x|^{N+ps}}dxdy\right)^{1/p}\nonumber\\
		&\times\left(\int_{\mathbb{R}^{N}\setminus B_r(x_0)}\frac{1}{|x_0-x|^{N+ps}}dxdy\right)^{1/q}\nonumber\\
		&\leq Cr^{s}\left(\int_{\mathbb{R}^N}\frac{|u(x_0)-u(x)|^p}{|x-x_0|^{N+ps}}dx\right)^{1/p}\rightarrow 0~\text{as}~r\rightarrow 0^{+}.
		\end{align}
		Thus we have 
		\begin{align}\label{5.15}
		\lim_{r\rightarrow 0^{+}}\operatorname{ess}\sup_{B_{r/2}(x_0)}(-u)&\leq -u(x_0).
		\end{align}
		This implies $u^{*}(x_0)\geq u(x_0)$ for every Lebesgue point in $\Omega$ and hence for almost all $x\in\Omega$. From \eqref{5.11} and \eqref{5.15}, we obtain the lower semicontinuous representation of $u$.
	\end{proof}
	\begin{corollary}[Strong Maximum Principle]\label{strong maximum principle}
		Let $u\in X_0$. Suppose $u\geq 0\in\mathbb{R}^N\setminus\Omega$ and for all $\phi\in X_0$ with $\phi\geq0$, we have
		\begin{equation*}
		\int_Q \cfrac{|u(x)-u(y)|^{p-2}(u(x)-u(y))}{|x-y|^{N+ps}}(\phi(x)-\phi(y))dxdy\geq0.
		\end{equation*}
		Then $u$ has a lower semicontinuous representation in $\Omega$, such that either $u\equiv0$ or $u>0$.
	\end{corollary}
	\begin{proof}
		See Lemma 2.3 of \cite{mosconi2016nonlocal}.
	\end{proof}
	\subsection{Existence of weak solutions}
	We begin the section by considering the problem
	\begin{align}\label{squasinna}
	(-\Delta_p)^s w&=\lambda w^{-\gamma}~\text{in}~\Omega,\nonumber\\
	w&>0~\text{in}~\Omega,\nonumber\\
	w&=0~\text{in}~\mathbb{R}^N\setminus\Omega.
	\end{align}
	We now state an existence result due to \cite{canino2017nonlocal}.
	\begin{lemma}
		Assume $0<\gamma<1$ and $\lambda>0$. Then the problem \eqref{squasinna} has a unique solution, $\underline{u}_{\lambda}\in W_0^{1, p}(\Omega)$, such that for every $K\subset\subset\Omega$, $ess.\underset{K}{\inf}\,\underline{u}_\lambda>0$.
	\end{lemma}
	\noindent With this Lemma in consideration, we now prove our first major theorem.
	\begin{lemma}\label{lambda finite}
		Assume $0<\gamma<1<q \leq p_s^{*}-1$. Then $0<\varLambda<\infty$.
	\end{lemma}
	\begin{proof}
		Define,
		\[   
		\overline{f}(x,t) = 
		\begin{cases}
		f(t), &~\text{if}~t>\overline{u}_{\lambda}\\
		f(\underline{u}_{\lambda}),&~\text{if}~t\leq \underline{u}_{\lambda}
		\end{cases}\]
		where, $f(u)=\frac{\lambda}{u^{\gamma}}+u^q$ and $\underline{u}_{\lambda}$ is the solution to \eqref{squasinna}. Let $\overline{F}(x,s)=\int_{0}^{s}\overline{f}(x,t)ds$, $\lambda>0$. 
		Define a function $\overline{I}_{\lambda}:X_0\rightarrow\mathbb{R}$ as follows.
		\begin{align}\label{fang_fnal}
		\overline{I}_{\lambda}(u)&=\frac{1}{p}\int_{Q}\frac{|u(x)-u(y)|^p}{|x-y|^{N+ps}}dxdy-\int_{\Omega}\overline{F}(x,u)dx.
		\end{align}
		The functional is $C^1$ (refer to Lemma \ref{C1} in the Appendix) and weakly lower semicontinuous. From the H\"{o}lder's inequality and Lemma (\ref{embedding}), we obtain
		\begin{align}
		\overline{I}_{\lambda}(u)&=\frac{1}{p}\|u\|^p-\int_{\Omega}\overline{F}(x,u)dx\nonumber\\
		&\geq \frac{1}{p}\|u\|^p-\lambda c_1\|u\|^{1-\gamma}-c_2\|u\|^{q+1}\nonumber
		\end{align}
		where, $c_1$, $c_2$ are constants. We choose $r>0$ small enough and $\lambda>0$ sufficiently small so that the term $\frac{1}{2}\|u\|^2-\lambda c_1\|u\|^{1-\gamma}-c_2\|u\|^{q+1}>0$. Thus we have a pair of $(\lambda, r)$ such that 
		\begin{align}
		\min_{u\in\partial B_r}\{\overline{I}_{\lambda}(u)\}&>0.\nonumber
		\end{align} 
		Now, for $\phi > 0\in X_0$, we have
		\begin{align}\label{inf_eqn}
		\overline{I}_{\lambda}(\phi)&=\frac{|t|^p}{p}\|\phi\|^p-\frac{|t|^{1-\gamma}}{1-\gamma}\int_{\Omega}|\phi|^{1-\gamma}dx-\frac{|t|^{q+1}}{q+1}\int_{\Omega}|\phi|^{q+1}dx.
		\end{align}
		The above equation (\ref{inf_eqn}) holds since $1-\gamma<1<q+1$. Thus $\overline{I}_{\lambda}(t\phi)\rightarrow -\infty$ as $t\rightarrow\infty$.  Therefore we have $\underset{||u||_{X_0}\leq r}{\inf}\overline{I}_{\lambda}(u)=c<0$. By the definition of infimum, we consider a minimizing sequence $\{u_n\}$ for $c$ . By the reflexivity of $X_0$ there exists a subsequence, still denoted by $\{u_n\}$, which weakly converges to, say, $u$.  such that 
		\begin{align}
		u_n&\rightarrow u~\text{weakly in}~L^{p_s^*}(\Omega)\nonumber\\
		u_n&\rightarrow u~\text{strongly in }~L^r(\Omega)~\text{for}~1\leq r < p_s^*\nonumber\\
		u_n&\rightarrow u~\text{pointwise a.e. in}~\Omega.
		\end{align}
		Therefore, by the Brezis-Lieb lemma \cite{brezis1983relation}, we get
		\begin{align}\label{lieb0}
		\|u_n\|^p&=\|u\|^p+\|u_n-u\|^p+o(1)\nonumber\\
		\|u_n\|_{q+1}^{q+1}&=\|u\|_{q+1}^{q+1}+\|u_n-u\|^{q+1}+o(1).
		\end{align}
		On using the H\"{o}lder's inequality and passing the limit $n\rightarrow\infty$, we obtain
		\begin{align}\label{lieb1}
		\int_{\Omega}u_n^{1-\gamma}dx&\leq\int_{\Omega}u^{1-\gamma}dx+\int_{\Omega}|u_n-u|^{1-\gamma}dx\nonumber\\
		&\leq \int_{\Omega}u^{1-\gamma}dx+c_1\|u_n-u\|_{p}^{1-\gamma}\nonumber\\
		&=\int_{\Omega}u^{1-\gamma}dx+o(1).
		\end{align}
		Therefore, on similar lines, we have
		\begin{align}\label{lieb2}
		\int_{\Omega}u^{1-\gamma}dx&\leq\int_{\Omega}u_n^{1-\gamma}dx+\int_{\Omega}|u_n-u|^{1-\gamma}dx\nonumber\\
		&\leq \int_{\Omega}u_n^{1-\gamma}dx+c_1\|u_n-u\|_{p}^{1-\gamma}\nonumber\\
		&=\int_{\Omega}u^{1-\gamma}dx+o(1).
		\end{align}
		Combining \eqref{lieb1} and \eqref{lieb2}, we obtain the following
		\begin{align}
		\int_{\Omega}u_n^{1-\gamma}dx&=\int_{\Omega}u^{1-\gamma}dx+o(1).
		\end{align}
		Thus, clubbing equations \eqref{lieb0}, \eqref{lieb1} and \eqref{lieb2}, we deduce that 
		\begin{align}\label{lieb4}
		\overline{I}_{\lambda}(u_n)&=\overline{I}_{\lambda}(u)+\frac{1}{p}\|u_n-u\|^p-\frac{1}{q+1}\|u_n-u\|_{q+1}^{q+1}+o(1).
		\end{align}
		We also observe from \eqref{lieb0} that  for $n$ sufficiently large $u$, $u_n-u\in B_r$ and $\frac{1}{p}\|u_n-u\|^p-\frac{1}{q+1}\|u_n-u\|_{q+1}^{q+1}\geq o(1)$. Since $r>0$ was chosen to be sufficiently small, we have
		\begin{align}
		\frac{1}{p}\|u_n-u\|^p-\frac{1}{q+1}\|u_n-u\|_{q+1}^{q+1}&> 0~\text{on}~\partial B_r\nonumber\\
		\frac{1}{p}\|u_n-u\|^p-\frac{1}{q+1}\|u_n-u\|_{q+1}^{q+1}&\geq 0~\text{in}~B_r.
		\end{align}
		As a consequence, we can conclude that 
		\begin{align}
		\frac{1}{p}\|u_n-u\|^p-\frac{1}{q+1}\|u_n-u\|_{q+1}^{q+1}&\geq o(1).
		\end{align}
		Therefore on passing the limit $n\rightarrow\infty$ to \eqref{lieb4}, we obtain $\overline{I}_{\lambda}(u_n)\geq \overline{I}_{\lambda}(u)+o(1)$. Since $\underset{\|u\|_{X_0}\leq r}{\inf}\overline{I}_{\lambda}(u)=c$ we have $u\neq 0$ which is a minimizer of $\overline{I}_{\lambda}$ over $X_0$. Thus we have 
		\begin{align}
		(-\Delta_p)^s u&=f(x,u)~\text{in}~\Omega,\nonumber\\
		u&>0~\text{in}~\Omega,\nonumber\\
		u&=0~\text{in}~\mathbb{R}^N\setminus\Omega.
		\end{align}  
		By the comparison principle (refer to lemma \ref{weak comparison} in the Appendix) of fractional $p$-Laplacian we conclude that $\overline{u}_{\lambda}\leq u$ in $\Omega$. Thus $\Lambda>0$ since the choice $\lambda>0$ has been made.\\
		We now claim that $\varLambda<\infty$. We let $\lambda_1$ to denote the principal eigenvalue of $(-\Delta_p)^s$ in $\Omega$ and let $\phi_1>0$ be the associated eigenfunction. In other words, we have
		\begin{align}
		(-\Delta_p)^s\phi_1&=\lambda_1|\phi_1|^{p-2}\phi_1~\text{in}~\Omega,\nonumber\\
		\phi_1&>0~\text{in}~\Omega,\nonumber\\
		\phi&=0~\text{in}~\mathbb{R}^N\setminus\Omega. 
		\end{align}      
		We choose, $\phi_1$ as a test function in the weak formulation of \eqref{main}, to get
		\begin{align}\label{contradiction}
		\lambda_1\int_{\Omega}u|\phi_1|^{p-2}\phi_1dx&=\int_{\Omega}(-\Delta_p)^s\phi_1 u dx\nonumber\\
		&=\int_{\Omega}\left(\frac{\lambda}{u^{\gamma}}+u^q\right)\phi_1dx.
		\end{align}
		Let $\tilde{\varLambda}$ be any constant such that $\tilde{\Lambda}t^{-\gamma}+t^q>p\lambda_1 t$, $\forall t>0$. This leads to a contradiction to the equation \eqref{contradiction}. Hence we conclude that $\Lambda<\infty$.
	\end{proof}
	\begin{lemma}\label{subsuperlemma}
		Let $0<\gamma<1$. Suppose that $\underline{u}$ is a weak subsolution while $\overline{u}$ is a weak supersolution to the problem \eqref{main} such that $\underline{u}\leq\overline{u}$, then for every $\lambda\in0,\Lambda)$ there exists a weak solution $u_{\lambda}$ such that $\underline{u}\leq u_{\lambda} \leq\overline{u}$ a.e. in $\Omega$. This $u_{\lambda}$ is a local minimizer of $\overline{I}_{\lambda}$ defined over $X_0$.
	\end{lemma}
	\begin{proof}
		We begin by showing that  $\underline{u}\leq\overline{u}$. For this let us consider the problem \eqref{main}. Let $\mu\in(0,\Lambda)$. By the definition of $\Lambda$, there exists $\lambda_0\in(\mu,\Lambda)$ such that \eqref{main} with $\lambda=\lambda_0$ has a solution by the Lemma \ref{lambda finite}, say $u_{\lambda_0}$. Then $\overline{u}=u_{\lambda_0}$ happens to be a supersolution of the problem \eqref{main}. Consider the function $\phi_1$ an eigenfunction of $(-\Delta_p)^s$ corresponding to the smallest eigenvalue $\lambda_1$. Thus $\phi_1\in L^{\infty}(\Omega)$ \cite{lindgren2014fractional} and 
		\begin{align}
		(-\Delta_p)^s\phi_1&=\lambda_1|\phi_1|^{p-2}\phi_1,\nonumber\\ \phi_1&>0~\text{in}~\Omega,\nonumber\\
		\phi_1&=0~\text{in}~\mathbb{R}^N\setminus\Omega.
		\end{align}
		Choose, $t>0$ such that $t\phi_1\leq \overline{u}$ and $t^{p+q-1}\phi_1^{p+q-1}\leq \frac{\lambda}{\lambda_1}$. On defining $\underline{u}=t\phi_1$ we have
		\begin{align}
		(-\Delta_p)^s\underline{u}&=\lambda_1t^{p-1}\phi_1^{p-1}\nonumber\\
		&\leq \lambda t^{-q}\phi_1^{-q}+t^{\alpha}\phi_1^{\alpha}\nonumber\\
		&=\lambda \underline{u}^{-q}+\underline{u}^{\alpha},
		\end{align}
		i.e., $\underline{u}$ is a subsolution of the problem \eqref{main}. This implies that $\underline{u}\leq\overline{u}$. \\
		We now show the existence of a $u_{\lambda}$. For this, we define 
		\[   
		\tilde{f}(x,t) = 
		\begin{cases}
		f_\lambda(x, \overline{u}), &~\text{if}~t\geq\overline{u}\\
		f_\lambda(x, t), &~\text{if}~\underline{u}\leq t\leq\overline{u}\\
		f_\lambda(x, \underline{u}),&~\text{if}~t\leq \underline{u}.
		\end{cases}\]
		We further define $\tilde{I}(u)=\frac{1}{p}\|u\|^p-\int_{\Omega}\tilde{F}(x,u)dx$, where $\tilde{F}(x,t)=\int_{0}^{t}\tilde{f}(x,s)ds$. Let $u_\lambda$ be a global minimizer of the functional $\tilde{I}$ due to the definition of $\tilde{f}$. We first observe that the $C^1$ functional $\tilde{I}$ is sequentially weakly lower semicontinuous and coercive. This can be seen from the dominated convergence theorem and the Sobolev embedding. Due to the monotonicity of $\tilde{f}$ we have,
		\begin{align}
		(-\Delta_p)^s(\overline{u}-u_\lambda)&\geq f(x,\overline{u})-\tilde{f}(x,\overline{u})\nonumber\\
		&\geq 0~\text{in}~\Omega
		\end{align}
		\noindent along with $\overline{u}-u_\lambda\geq 0$, in $\mathbb{R}^N\setminus\Omega$.  We now refer to a result proved in the Lemma \ref{strong} in the Appendix that $\overline{u}-u_\lambda\geq 0$ and is a weak supersolution to the problem \eqref{main}. On using the Lemma 2.7 in \cite{iannizzotto2015hs}, we have $\frac{\bar{u}-u_\lambda}{\delta^s}\geq c_2>0$ in $\bar{\Omega}$. Similarly we prove $\frac{u_\lambda-\underline{u}}{\delta^s}\geq c_2>0$ in $\bar{\Omega}$. Then $u_\lambda$ is a weak solution to the problem \eqref{main}.\\
		Now, we prove that $u_\lambda$ is a local minimizer of $\overline{I}_\lambda$. Due to Theorem 4.4 in \cite{iannizzotto2014global}, we have $u_\lambda\in C_{s}^0(\Omega)$. Thus for any $u\in B_{\frac{c_2}{2}}^{\delta}(u_\lambda)$, we obtain $\frac{\overline{u}-u}{\delta^s}=\frac{\overline{u}-u_\lambda}{\delta^s}+\frac{u_\lambda-u}{\delta^s}\geq c_2-\frac{c_2}{2}$ in $\bar{\Omega}$. Hence, by the maximum principle we get $\overline{u}-u_\lambda>0$ in $\Omega$. using similar argument, it follows that $\underline{u}-u_\lambda>0$ in $\Omega$. Therefore, $\tilde{I}$ and $\bar{I}_\lambda$ becomes identical over $B_{\frac{c_2}{2}}^{\delta}(u_\lambda)\cap X_0$. Further we have $u_\lambda$ is local minimizer of $\bar{I}_\lambda$ in $C_{s}^0(\Omega)\cap X_0$. Hence, Theorem 1.1 in \cite{iannizzotto2015hs}, implies that $u_\lambda$ is a local minimizer of $\bar{I}_\lambda$.
	\end{proof}
	\noindent We now prove the following theorem.
	\begin{theorem}\label{solnLambda}
		The problem in \eqref{main} has at least one solution if $\lambda=\Lambda$.
	\end{theorem}                    
	\begin{proof}
		Consider an increasing sequence $\{\lambda_n\}$, which converges to $\Lambda$, as $n\rightarrow\infty$. Let $u_n=u_{\lambda_n}$ be a weak solution to the problem \eqref{main} for $\lambda=\lambda_n$. Thus 
		\begin{align}\label{bdd1}
		\int_{Q}&\frac{|u_n(x)-u_n(y)|^{p-2}(u_n(x)-u_n(y))(\phi(x)-\phi(y))}{|x-y|^{N+ps}}dxdy\nonumber\\
		&\hspace{1cm}-\lambda_n\int_{\Omega}u_n^{-\gamma}\phi dx-\int_{\Omega}u_n^q\phi dx=0,~\forall~\phi\in X_0.
		\end{align}   
		Hence putting $\phi=u_n$, we have
		\begin{align}\label{bdd2}
		\int_{Q}\frac{|u_n(x)-u_n(y)|^{p}}{|x-y|^{N+ps}}dxdy&-\lambda_n\int_{\Omega}u_n^{1-\gamma} dx-\int_{\Omega}u_n^{q+1}dx=0.
		\end{align}   
		From the Lemma \ref{subsuperlemma}, the energy functional
		\begin{align}\label{energy}
		I(u_n)&=\frac{1}{p}\int_{Q}\frac{|u_n(x)-u_n(y)|^{p}}{|x-y|^{N+ps}}dxdy-\frac{\lambda_n}{1-\gamma}\int_{\Omega}u_n^{1-\gamma} dx-\frac{1}{q+1}\int_{\Omega}u_n^{q+1}dx\nonumber\\
		&\leq B
		\end{align}
		for every $0<\gamma<1$.  Using \eqref{bdd2} in \eqref{energy} we get
		\begin{align}\label{3.21}
		\frac{1}{p}\left(\lambda_n\int_{\Omega}u_n^{1-\gamma} dx+\int_{\Omega}u_n^{q+1}dx\right)-\frac{\lambda_n}{1-\gamma}\int_{\Omega}u_n^{1-\gamma} dx-\frac{1}{q+1}\int_{\Omega}u_n^{q+1}dx\leq B.
		\end{align}
		From \eqref{3.21} we get 
		\begin{align}\label{3.22}
		\left(\frac{1}{p}-\frac{1}{q+1}\right)\int_{\Omega}u_n^{q+1}dx&\leq B+\lambda_n\left(\frac{1}{1-\gamma}-\frac{1}{p}\right)\int_{\Omega}u_n^{1-\gamma}dx.
		\end{align}
		Using \eqref{3.22} in \eqref{bdd2} we obtain
		\begin{align}\label{finalineq}
		\|u_n\|_{X_0}^{p-1+\gamma}&\leq A_1+\frac{A_2}{\|u_n\|_{X_0}^{1-\gamma}}.
		\end{align}
		From the inequality in \eqref{finalineq}, it is easy to see that $\underset{{n\in\mathbb{N}}}{\sup}\|u_n\|_{X_0}<\infty$. Thus by the reflexivity of $X_0$, we have a subsequence, which will still be denoted by $\{u_n\}$, such that $u_n\rightharpoonup u$ weakly in $X_0$, as $n\rightarrow\infty$. This establishes that $u$ is a weak solution corresponding to $\Lambda$. 
	\end{proof}
	\noindent We now prove a corollary based on the Theorem \ref{solnLambda}.
	\begin{corollary}
		Let $1<q\leq p_s^{*}-1$, $0<\gamma <1$ and $0<\lambda\leq\Lambda$. Then there exists a smallest solution in $X_0$ to the problem \eqref{main}.
	\end{corollary}
	\begin{proof}
		From Lemma \ref{subsuperlemma}, we guarantee the existence of a weak solution $u_\lambda$ to the problem \eqref{main} for $\lambda\in (0,\Lambda)$. We now define a sequence $\{v_n\}$ by the following iterative sequence of problems. Define $v_1=\underline{u}$, a subsolution of \eqref{main}. The remaining terms of the sequence can be defined by the following iterative scheme.
		\begin{align}
		(-\Delta_p)^s v_n-\frac{\lambda}{v_n^{\gamma}}&=v_{n-1}^{q}~\text{in}~\Omega,\nonumber\\
		v_n&>0~\text{in}~\Omega,\nonumber\\
		v_n&=0~\text{in}~\mathbb{R}^{N}\setminus\Omega,
		\end{align} 
		for each $n\in\mathbb{N}$. By the choice of $v_1$ we have $v_1\leq u$, where $u$ is a weak solution to \eqref{main}, whose existence is again attributed to the Lemma \ref{subsuperlemma}. By the weak comparison principle (refer Lemma \ref{weak comparison} in the Appendix), it is clear that $v_1\leq v_2\leq...\leq u$. Owing to the Theorem 6.4 in \cite{mukherjee2016fractional}, we have $u$ is in $L^{\infty}(\Omega)$, which further implies that the sequence $\{v_n\}$ is bounded in $X_0$. Thus we have a subsequence such that $v_n\rightharpoonup \hat{u}$. To conclude that $\hat{u}$ is the minimal solution, we let $\hat{v}$ to be a solution to \eqref{main}. We have $v_n\leq\hat{v}$ which on passing the limit $n\rightarrow\infty$ we get $\hat{u}\leq\hat{v}$.   
	\end{proof}
	\subsection{$C^1$ versus $W^{s,p}$ local minimizers of the energy} The following lemma is useful to prove the  multiplicity of solutions. More precisely,  
	we now turn our attention to the main theorem.
	\begin{theorem}\label{versus}
		With the growth conditions of f in tact.  Let $u_0\in C^1(\overline{\Omega})$satisfying 
		\begin{equation}\label{bord}
		u_0\geq \eta\mbox{d}(x,\partial\Omega)\mbox{ for some }\eta>0
		\end{equation}
		be a local minimizer of $I$ in $C^1(\overline{\Omega})$ topology; that is,
		\begin{eqnarray*}
			\exists\,\epsilon >0\,\mbox{such that }\,u\in C^1(\overline{\Omega})\;, 
			\|u-u_0\|_{C^1(\overline{\Omega})}<\epsilon\Rightarrow I(u_0)\leq I(u).
		\end{eqnarray*}
		Then, $u_0$ is a local minimum of $I$ in $W^{s,p}_0(\Omega)$ also.
	\end{theorem}
	For proving Theorem \ref{versus}, we will need uniform $L^\infty$-estimates for a family of solutions to $(P_\epsilon)$  as below.
\begin{theorem}\label{linftyestimate}
Let $\{u_\epsilon\}_{\epsilon\in(0,1)}$ be a family of solutions to ${\rm (P_\epsilon)}$, where $u_0$ satisfies \eqref{bord} and solves $\rm{(P)}$; let $\displaystyle\sup_{\epsilon\in (0,1)}(\|u_\epsilon\|_{W^{s,p}_0(\Omega)})<\infty$. Then, there exists $C_1, C_2>0$ (independent of $\epsilon$) such that
\begin{eqnarray*}
\displaystyle\sup_{\epsilon\in(0,1)}\|u_\epsilon\|_{L^\infty(\Omega)}<\infty\;\mbox{and}\; C_1 d(x,\partial\Omega)\leq u_\epsilon\leq C_2 d(x,\partial\Omega).
\end{eqnarray*}
\end{theorem}
The proof of the above theorem is a consequence of the results proved in section \label{regularity}.
	\begin{proof}
		Let $K\Subset\Omega$. We divide the proof into two cases.\\
		{\bf Case 1}:~{\it subcritical case} when $q<p_s^*-1$.\\
		We prove by contradiction, i.e. suppose $u_0$ is not a local minimizer. Let $r\in (q,p_s^*-1)$ and define
		\begin{align}\label{aux1}
		\begin{split}
		K(w)&=\frac{1}{r+1}\int_{K}|w-u_0|^{r+1}dx, (w\in W^{s,p}(K)).
		\end{split}
		\end{align}
		{\it Case i}:~{$K(v_{\epsilon})<\epsilon$}.\\
		Define $S_{\epsilon}=\{v\in W_0^{s,p}(\Omega):0\leq K(v)\leq \epsilon\}$. Consider the problem $I_{\epsilon}=\inf_{v\in S_{\epsilon}}\{I(v)\}$. The infimum exists since the set $S_{\epsilon}$ is bounded and the functional $I$ is $C^1$. In addition, $I$ is also w.l.s.c. and $S_{\epsilon}$ is closed, convex. Thus $I_{\epsilon}$ is actually attained, at say $v_{\epsilon}\in S_{\epsilon}$, and $I_{\epsilon}=I(v_{\epsilon})<I(u_0)$. \\
		{\it Claim}:~We will now show that  $\exists \eta>0$ such that $v_{\epsilon}\geq \eta \phi_1$.\\
		{\it Proof}:~ The proof is again by contradiction., i.e. $\forall\eta>0$ let $|\Omega_{\eta}|=|supp\{(\eta\phi_1-v_{\epsilon})^{+}\}|>0$. Define $v_{\eta}=(\eta\phi_1-v_{\epsilon})^{+}$. For $0<t<1$ define $\xi(t)=I(v_{\epsilon}+v_{\eta})$. Thus 
		\begin{align}
		\begin{split}
		\xi'(t)&=\langle I'(v_{\epsilon}+tv_{\eta}),v_{\eta}\rangle\\
		&=\langle(-\Delta_p)^s(v_{\epsilon}+tv_{\eta})-(v_{\epsilon}+tv_{\eta})^{-\alpha}-f(x,v_{\epsilon}+tv_{\eta}),v_{\eta}\rangle.
		\end{split}
		\end{align}
		Similarly,
		\begin{align}
		\begin{split}
		\xi'(1)&=\langle I'(v_{\epsilon}+v_{\eta}),v_{\eta}\rangle\\
		&=\langle I'(\eta\phi_1),v_{\eta}\rangle\\
		&=\langle(-\Delta_p)^s(\eta\phi_1)-(\eta\phi_1)^{-\alpha}-f(x,\eta\phi_1),v_{\eta}\rangle<0
		\end{split}
		\end{align}
		for sufficiently small $\eta>0$. Moreover,
		\begin{align}
		\begin{split}
		-\xi'(1)+\xi'(t)&=\langle(-\Delta_p)^s(v_{\epsilon}+tv_{\eta})-(-\Delta_p)^s(v_{\epsilon}+v_{\eta})\\
		&+((v_{\epsilon}+v_{\eta})^{-\alpha}-(v_{\epsilon}+tv_{\eta})^{-\alpha})+(f(x,v_{\epsilon}+v_{\eta})-f(x,v_{\epsilon}+tv_{\eta})),v_{\eta}\rangle.
		\end{split}
		\end{align}
		Since $s^{-\alpha}+f(x,s)$ is a uniformly nonincreasing function with respect to $x\in\Omega$ for sufficiently small $s>0$. Also from the monotonicity of $(-\Delta_p)^s$ we have, for sufficiently small $\eta>0$, $0\leq \xi'(1)-\xi'(t)$. From the Taylor series expansion and the fact that $K(v_{\epsilon})<\epsilon$ we have $\exists 0<\theta<1$ such that 
		\begin{align}
		\begin{split}
		0&\leq I(v_{\epsilon}+v_{\eta})-I(v_{\epsilon})\\
		&=\langle I'(v_{\epsilon}+\theta v_{\eta}),v_{\eta}\rangle\\
		&=\xi'(\theta).
		\end{split}
		\end{align}  
		Thus for $t=\theta$ we have $\xi'(\theta)\geq 0$ which is a contradicton to $\xi'(\theta)\leq\xi'(1)<0$ as obtained above. Thus $v_{\epsilon}\geq \eta\phi_1$ for some $\eta>0$. In fact from the Lemma \ref{calpha1}, \ref{calpha2} we have $ \sup_{\epsilon\in(0,1]}\{\|u_{\epsilon}\|_{C^{1,\alpha}(\overline{\Omega})}\}\leq C$. By the compact embedding $C^{1,\alpha}(\overline{\Omega})\hookrightarrow C^{1,\kappa}\overline{\Omega})$, for any $\kappa<\alpha$,  we have $u_{\epsilon}\rightarrow u_0$ which contradicts the assumption made.\\
		{\it Case ii}:~{$K(v_{\epsilon})=\epsilon$}\\
		In this case, from the Lagrange multiplier rule we have $I'(v_{\epsilon})=\mu_{\epsilon}K(v_{\epsilon})$. We will first show that $\mu_{\epsilon}\leq 0$. Suppose $\mu_{\epsilon}>0$, then $\exists \phi\in X$ such that 
		\begin{align*}
		\begin{split}
		\langle I'(v_{\epsilon}),\phi \rangle<0~ \text{and} ~\langle K'(v_{\epsilon}),\phi \rangle<0.
		\end{split}
		\end{align*}
		Then for small $t>0$ we have 
		\begin{align*}
		\begin{split}
		I(v_{\epsilon}+t \phi)&<I(v_{\epsilon})\\
		K(v_{\epsilon}+t \phi)&<K(v_{\epsilon})=\epsilon
		\end{split}
		\end{align*}
		which is a contradiction to $v_{\epsilon}$ being a minimizer of $I$ in $S_{\epsilon}$.\\
		{\it Case i}:~($\mu_{\epsilon}\in(-l,0)$ where $l>-\infty$).\\
		Consider the sequence of problems
		\begin{align}
		\begin{split}
		(P_{\epsilon}):~(-\Delta_p)^su_{\epsilon}&=\gamma(\mu_{\epsilon},x,t)
		\end{split}
		\end{align}
		where $ \gamma(\mu_{\epsilon},x,t)=t^{-\alpha}+f(x,t)+\mu_{\epsilon}|t-u_0|^{r-1}(t-u_0)$. From the weak comparison principle we have $v_{\epsilon}\leq \eta\phi_1$ for some $\eta>0$ small enough, independent of $\epsilon$.  This is beacuse $\eta\phi_1$ is a strict subsolution to $P_{\epsilon}$. Further, since $-l\leq \mu_{\epsilon}\leq 0$, there exists $M$, $c$ such that 
		\begin{align}
		\begin{split}
		(-\Delta_p)^s(v_{\epsilon}-1)^{+}&\leq M+c((v_{\epsilon}-1)^{+})^r.
		\end{split}
		\end{align}
		Using the Moser iteration technique as in Lemma \ref{bounded} we obtain $\|v_{\epsilon}\|_{\infty}\leq C'$. Therefore $\exists L>0$ such that $\eta\phi_1\leq v_{\epsilon}\leq L\phi_1$. By using the arguments previously used, we end up getting $|v_{\epsilon}|_{C^{\alpha}(\overline{\Omega})}\leq C'$. The conclusion follows as in the previous case of $K(v)<\epsilon$.\\
		{\it Case ii}:~$\inf_{\mu_{\epsilon}}=-\infty$\\
		Let us assume $\mu_{\epsilon}\leq -1$. As above, we can similarly obtain $v_{\epsilon\phi_1}$ for $\eta>0$ small enough and independent of $\epsilon$. Further, there exists a constant $M>0$ such that $\gamma(s,x,t)<0$ $\forall (s,x,t)\in(-\infty,-1]\times\Omega\times(M,\infty)$.\\
		Then from the weak comparison principle on $(-\Delta_p)^s$, we gate $v_{\epsilon}\leq M$ for $\epsilon>0$ sufficiently small. Since $u_0$ is a local $C^1$ -  minimizer, $u_0$ is a weak solution to (P) and hence
		\begin{align}
		\begin{split}
		\langle (-\Delta_p)^su_0,\phi \rangle&=\int_{\Omega}u_0^{-\alpha}\phi dx+\int_{\Omega}f(x,u_0)\phi dx
		\end{split}
		\end{align}
		$\forall\phi\in C_c^{\infty}(\Omega)$. In fact from lemma (to be written next) we have for every function $w\in W_0^{s,p}(\Omega)$, $u_0$ satisfies 
		\begin{align}\label{s1}
		\begin{split}
		\langle (-\Delta_p)^su_0,w \rangle&=\int_{\Omega}u_0^{-\alpha}w dx+\int_{\Omega}f(x,u_0)w dx.
		\end{split}
		\end{align}
		Similarly,
		\begin{align}\label{s2}
		\begin{split}
		\langle (-\Delta_p)^sv_{\epsilon},wi \rangle&=\int_{\Omega}v_{\epsilon}^{-\alpha}w dx+\int_{\Omega}f(x,v_{\epsilon})w dx.
		\end{split}
		\end{align}
		On subtracting the relations, i.e. (\ref{s2})-(\ref{s1}) and testing with $w=|v_{\epsilon}-u_0|^{\beta}-1(v_{\epsilon}-u_0)$, where $\beta\geq 1$, we obtain
		\begin{align}
		\begin{split}
		0&\leq \beta\langle (-\Delta_p)^sv_{\epsilon}-(-\Delta_p)^su_0,|v_{\epsilon}-u_0|^{\beta}-1(v_{\epsilon}-u_0) \rangle\\
		&-\int_{\Omega}(g(v_{\epsilon}-g(u_0)))|v_{\epsilon}-u_0|^{\beta}-1(v_{\epsilon}-u_0)dx\\
		&=\int_{\Omega}(f(x,v_{\epsilon})-f(x,u_0))|v_{\epsilon}-u_0|^{\beta}-1(v_{\epsilon}-u_0)dx\\
		&+\mu_{\epsilon}\int_{\Omega}|v_{\epsilon}-u_0|^{\beta+r}dx.
		\end{split}
		\end{align}
		By H\"{o}lder's inequality and the bounds of $v_{\epsilon}$, $u_0$ we get
		\begin{align}
		\begin{split}
		-\mu_{\epsilon}\|v_{\epsilon}-u_0\|_{\beta+r}^{r}&\leq C|\Omega|^{\frac{r}{\beta+r}}.
		\end{split}
		\end{align}
		Here $C$ is independent of $\epsilon$ and $\beta$. On passing the limit $\beta\rightarrow\infty$ we get $-\mu_{\epsilon}\|v_{\epsilon}-u_0\|_{\infty}\leq C$. Working on similar lines we end up getting $v_{\epsilon}$ is bounded in $C^{\alpha}(\overline{\Omega})$ independent of $\epsilon$ and the conclusion follows. This marks an end to the subcritical case.\\\\
		{\bf Case 2}:~{\it $q=p_s^{*}-1$}\\
		The proof again follows by contradiction, i.e. we assume that the conclusion of the Theorem is untrue. Let 
		\begin{align}
		\begin{split}
		\chi(w)&=\frac{1}{p_s^*}\int_{\Omega}|w-u_0|^{p_s^*}dx, w\in W_0^{s,p}(\Omega).
		\end{split}
		\end{align}
		Further define 
		\begin{align*}
		\begin{split}
		C_{\epsilon}&=\{v\in W_0^{s,p}(\Omega):\chi(v)\leq\epsilon\}.
		\end{split}
		\end{align*}
		Further, consider the truncated functional 
		\begin{align*}
		\begin{split}
		I_j(v)&=\frac{1}{p}\|v\|_{X}^p-\int_{\Omega}G(v)dx-\int_{\Omega}F_j(x,v)dx, \forall v\in W_0^{s,p}(\Omega)
		\end{split}
		\end{align*}
		where $f_j(x,u)=f(x,T_j(s))$ and 
		\[T_j(s)=\begin{cases}
		-j & s\leq -j \\
		s & -j \leq s \leq j   \\
		j  & s\geq j.          
		\end{cases}\] 
		By the `{\it Lebesgue theorem}' we have, for any $v\in W_0^{s,p}(\Omega)$, $I_j(v)\rightarrow I(v)$ as $j\rightarrow\infty$. It follows from the truncation and this convergence that for each $\epsilon>0$, there is some $j_{\epsilon}$ (with $j_{\epsilon}\rightarrow\infty$ as $\epsilon\rightarrow 0^+$) such that $I_{j_{\epsilon}}(v_{\epsilon})\leq I(v_{\epsilon})\leq I(u_0)$.\\
		On the other hand, since $C_{\epsilon}$ is closed, convex and since $I_{j_{\epsilon}}$ is weakly lower semicontinuous we deduce that $I_{j_{\epsilon}}$ achieves its infimum at some $u_{\epsilon}\in C_{\epsilon}$. Therefore, for $\epsilon>0$ small enough, we have 
		\begin{align*}
		\begin{split}
		I_{j_{\epsilon}}(u_{\epsilon})&\leq I_{\epsilon}(v_{\epsilon})<I(u_0).
		\end{split}
		\end{align*}
		Again, we have $u_{\epsilon}\geq \eta\phi_1$ for $\eta>0$ sufficiently small, independent of $\epsilon$.
		By Lagrange multiplier, there exists $\mu_{\epsilon}\leq 0$ such that $I'(u_{\epsilon})=\mu_{\epsilon}\chi'(u_{\epsilon})$. By the construction we have $u_{\epsilon}\rightarrow u_0$ as $\epsilon\rightarrow 0$ in $L^{p_s^*}(\Omega)$.\\
		{\it Claim}:~$(u_{\epsilon})$ is bounded in $L^{\infty}(\Omega)$ as $\epsilon\rightarrow 0$.\\
		{\it Case i}:~$\inf_{0<\epsilon<1}\{\mu_{\epsilon}\}>-\infty$.\\
		Look at 
		\begin{align*}
		\begin{split}
		(P_{\epsilon}):~(-\Delta_p)^su&=u^{-\alpha}+f_{j_{\epsilon}}(x,u)+\mu_{\epsilon}|u-u_0|^{p_s^*-2}(u-u_0).
		\end{split}
		\end{align*}
		which is satisfied weakly by $u_{\epsilon}$. Similar to the argument in the subcritical case we find $M>0$ independent of $\epsilon$, such that, 
		\begin{align*}
		\begin{split}
		(-\Delta_p)^s(u_{\epsilon}-1)^{+}&\leq M+c|(u_{\epsilon}-1)^{+}|^{p_s^*-1}.
		\end{split}
		\end{align*}
		By the Moser iteration method we get $(u_{\epsilon})$ is bounded in $L^{\infty}(\Omega)$. Going by the steps in the subcritical case we conclude Case i.\\
		{\it case ii}:~$\inf_{0<\epsilon<1}\{\mu_{\epsilon}\}=-\infty$\\
		By similar argument as used earlier we get $u_{\epsilon\geq \eta\phi_1}$ for some $\eta>0$ independent of $\epsilon$. Moreover, there exists $M>0$, independent of $\epsilon$, such that 
		\begin{align*}
		\begin{split}
		s^{-\alpha}+f_{j_{\epsilon}}(x,s)+\mu_{\epsilon}|s-u_0(x)|^{p_s^*-2}(s-u_0(x))&<0,~\text{if}~s>M.
		\end{split}
		\end{align*}
		Taking $(u_{\epsilon}-M)^{+}$ as a test function, we conclude that $u_{\epsilon}\leq M$ in $\Omega$. Continuing the proof as in the subcritical case we prove the claim of Case ii.
	\end{proof}
	\noindent{\bf Remark}:~Note that in all the above cases, since we already have $u_{\epsilon}\geq\eta\phi_1$ for $\eta>0$ sufficiently large and $u_{\epsilon}$ is an $L^{\infty}$-function, then there exists $L$ sufficiently large such that $\eta\phi_1\leq u_{\epsilon} \leq L\phi_1$. Thus $\eta\leq\frac{u_{\epsilon}}{\phi_1}\leq L$.
	\section{Multiplicity of weak solutions}
	This section is devoted to show the existence of a critical point $v_\lambda$  of the functional $\bar{I}_\lambda$ since, the functional $I_\lambda$ fails to be $C^1$. The critical point $v_\lambda$ of $\bar{I}_\lambda$ is also a point where the G\^{a}teaux derivative of the functional $I_\lambda$ vanishes. Therefore, $v_\lambda$ will solve the problem (\ref{main}). We will prove $v_\lambda\neq u_\lambda$, where, $u_\lambda$ is the solution to the problem \eqref{main} as proved in the Lemma \ref{subsuperlemma}. We have the following theorem proved in Ghoussoub-Preiss \cite{ghoussoub1989general}.
	\begin{theorem}[Ghoussoub-Preiss]\label{ghosub}
		Let $\varphi\colon X\rightarrow\mathbb{R}$ be a continuous and G\^{a}teaux differentiable function on a Banach space $X$ such that $\varphi\colon X\rightarrow X^*$ is continuous from the norm topology on $X$ to the weak$^*$ topology of $X^*$. Take two point $u_\lambda$ and $v_\lambda$ in $X$ and consider the number $$c={\underset{g\in\Gamma,}{\inf}}{\underset{0\leq t\leq1}{\max}} \varphi(g(t))$$
		where $\Gamma=\{g\in C([0,1], X)\colon g(0)=u_\lambda ~\&~ g(1)=v_\lambda\}$. Suppose $F$ is a closed subset of $X$ such that $F\cup\{x\in X: \varphi(x)\geq c\}$ seperates $u_\lambda$ and $v_\lambda$, then, there exists a sequence $\{x_n\}$ in $X$ verifying the following:
		\begin{itemize}
			\item[(i)] ${\underset{n \rightarrow\infty}{lim}}dist~(x_n,F)=0$
			\item[(ii)] $\underset{n \rightarrow\infty}{lim} \varphi (x_n)=c$
			\item[(iii)] $\underset{n \rightarrow\infty}{lim} \|\varphi^{'}(x_n)=0\|$
		\end{itemize}
	\end{theorem}
	\begin{definition}\label{ps defn}
		Let $F\subset\Omega$, be closed and $c\in\mathbb{R}$. Then a sequence $\{v_n\}\subset X_0$ is said be a Palais Smale sequence [in short $(PS)_{F,c}$] for the functional $\bar{I}_\lambda$ around $F$ at the level $c$, if
		$${\underset{n \rightarrow\infty}{lim}}dist~(x_n,F)=0,~~~ \underset{n \rightarrow\infty}{lim} \bar{I}_\lambda (x_n)=c~~~\&~~ \underset{n \rightarrow\infty}{lim} \|\bar{I}_\lambda{'}(x_n)\|=0$$
	\end{definition}
	\noindent Every $(PS)_{F,c}$ sequence for $\bar{I}_\lambda$ have the following compactness property.
	\begin{lemma}\label{4l3}
		Let $F\subset\Omega$ be closed and $c\in\mathbb{R}$. Let $\{v_n\}\subset X_0$ be a $(PS)_{F, c}$ sequence for the functional $\bar{I}_\lambda$, then $\{v_n\}$ is bounded in $X_0$ and there exists a subsequence $\{v_n\}$ such that $v_n\rightharpoonup v_\lambda$ in $X_0$, where $v_\lambda$ is a weak solution of the problem $(\ref{main}).$ 
	\end{lemma}
	\begin{proof}
		We use the Definition \ref{ps defn}, which says, there exists $K>0$ such that the following holds
		\begin{align*}
		\frac{1}{p}\int_{Q}\cfrac{|u_\lambda(x)-u_\lambda(y)|^{p}}{|x-y|^{N+ps}}dxdy-&\int_{v_n>\underline{u}_\lambda}\left[\left(\frac{\lambda}{1-\gamma}v_n^{1-\gamma}+\frac{v_n^{q+1}}{q+1}\right)-\left(\frac{\lambda}{1-\gamma}\underline{u}_\lambda^{1-\gamma}+\frac{\underline{u}_\lambda^{q+1}}{q+1}\right)\right]dx\\
		-&\int_{v_n\leq\underline{u}_\lambda}v_n(\lambda\underline{u}_\lambda^{-\gamma}+\underline{u}_\lambda^q)dx\leq K
		\end{align*}
		this implies,
		\begin{equation}\label{4e1}
		\frac{1}{p}\int_{Q}\cfrac{|u_\lambda(x)-u_\lambda(y)|^{p}}{|x-y|^{N+ps}}dxdy-\int_{v_n>\underline{u}_\lambda}\left(\frac{\lambda}{1-\gamma}v_n^{1-\gamma}+\frac{v_n^{q+1}}{q+1}\right)dx\leq K
		\end{equation}
		Again, by using the Definition \ref{ps defn}, we get
		\begin{align}\label{4e2}
		\frac{1}{p}\int_{Q}\cfrac{|u_\lambda(x)-u_\lambda(y)|^{p}}{|x-y|^{N+ps}}dxdy&=\int_{v_n>\underline{u}_\lambda}(\lambda v_n^{1-\gamma}+v_n^{q+1})dx +\int_{v_n\leq\underline{u}_\lambda}v_n (\lambda\underline{u}_\lambda^{-\gamma} +\underline{u}_\lambda^q)dx\nonumber\\ &+ o_n(1)\|v_n\|
		\end{align}
		Therefore, from \eqref{4e1} and \eqref{4e2}, we have
		\begin{align}\label{4e3}
		\|v_n\|^p + O_n(\|v_n\|) &\geq\int_{v_n>\underline{u}} v_n^{q+1}dx\nonumber\\
		&\geq\frac{q+1}{p}\|v_n\|^p -K
		\end{align}
		By using \eqref{4e3} we can conclude that the sequence $\{v_n\}$ is bounded in $X_0$. Since the space $X_0$ is reflexive, there exists $v_\lambda\in X_0$ such that $v_n\rightharpoonup v_\lambda$ in $X_0$ upto a subsequence. Thus 
		\begin{align*}
		&\int_{Q}\frac{|v_n(x)-v_n(y)|^{p-2}(v_n(x)-v_n(y))(\phi(x)-\phi(y))}{|x-y|^{N+ps}}dxdy\\
		&\rightarrow\int_{Q}\frac{|v_\lambda(x)-v_\lambda(y)|^{p-2}(v_\lambda(x)-v_\lambda(y))(\phi(x)-\phi(y))}{|x-y|^{N+ps}}dxdy,~~\forall\,\phi\in X_0.
		\end{align*}
		Passing the limit as $n\rightarrow\infty$ and applying the embedding result in Lemma $2.4$ we have, for $\phi\in X_0$
		\begin{equation}\label{4e4}
		\int_{Q}\frac{|v_\lambda(x)-v_\lambda(y)|^{p-2}(v_\lambda(x)-v_\lambda(y))(\phi(x)-\phi(y))}{|x-y|^{N+ps}}dxdy
		-\int_{\Omega}(\frac{\lambda}{v_\lambda^{\gamma}}-v_\lambda^q)\phi=0
		\end{equation}
		Therefore, using the strong maximum principle  and \eqref{4e4} we conclude that $v_\lambda$ is a weak solution of the problem \eqref{main}. This completes the proof.
	\end{proof}
	\noindent We observe from Lemma \ref{subsuperlemma} and the fact that $\bar{I}_\lambda(tu)\rightarrow -\infty$ as $t\rightarrow\infty$ for all $u\in X_0, u>0$, we can conclude that $\bar{I}_\lambda$ has a Mountain pass geometry near $u_\lambda$. Therefore, we may fix $e\in X_0, e>0$ such that $\bar{I}_\lambda(e)<\bar{I}_\lambda(u_\lambda)$. Let $R=\|e-u_\lambda\|$ and $r_0>0$ be small enough such that $u_\lambda$ is a minimizer of $\bar{I}_\lambda$ on $\overline{B(u_\lambda,r_0)}$. Consider the following complete metric space consisting of paths which is defined as
	\begin{equation*}
	\Gamma=\left\{\eta\in C\left(\left[0,\frac{1}{2}\right], X_0\right): \eta(0)=u_\lambda, \eta\left(\frac{1}{2}\right)=e\right\}
	\end{equation*}
	and the min-max value for mountain pass level
	$$\delta_0=\underset{\eta\in\Gamma,}{\inf}\underset{0\leq t\leq\frac{1}{2}}{\max}I_\lambda(\eta(t))$$
	Let us distinguish between the following two cases.
	\begin{enumerate}
		\item[] Case 1: (Zero altitude case). There exists $R_0>0$, such that
		\begin{align}\label{zero altitude}
		\inf\left\{\bar{I}_\lambda(\tilde{u}): \tilde{u}\in X_0 ,~ \|\tilde{u}-u_\lambda\|=r\right\}&\leq\bar{I}_\lambda(u_\lambda),~~\text{for all}~~ r<R_0.
		\end{align}
		\item[]  Case 2: There exists $r_1<r_0$ such that 
		\begin{align}\label{nonzero altitude}
		\inf\left\{\bar{I}_\lambda(\tilde{u}): \tilde{u}\in X_0 ~~\&~~ \|\tilde{u}-u_\lambda\|=r_1\right\}&>\bar{I}_\lambda(u_\lambda).
		\end{align}
	\end{enumerate}
	\begin{remark}
		Observe that, \eqref{zero altitude} implies $\delta_0=\bar{I}_\lambda(u_\lambda)$, whereas \eqref{nonzero altitude} implies $\delta_0>\bar{I}_\lambda(u_\lambda).$
	\end{remark}
	\noindent For the ``{\it Zero altitude case}" let us consider $F=\partial B(u_\lambda, r)$ with $r\leq R_0$. We can then construct a $(PS)_{F, \delta_0}$ sequence and get a second weak solution. We have the following result.
	\begin{lemma}
		Suppose Case 1 holds, then for $1<p<\infty, ~p-1<q\leq p_s^{*}-1, ~0<\gamma<1$ and $\lambda\in(0,\Lambda)$, there exists a weak solution $v_\lambda$ of the problem \eqref{main} such that $v_\lambda\neq u_\lambda$. 
	\end{lemma}
	\begin{proof}
		From Theorem \ref{ghosub} we can guarantee the existence of a $(PS)_{F, \delta_0}$ sequence $\{v_n\}$ for every $r\leq R_0$. From Lemma \ref{4l3}, we can conclude that the sequence $\{v_n\}$ is bounded in $X_0$ and it converges, upto a subsequence, to a weak solution $v_\lambda$ of the problem \eqref{main}. To show $v_\lambda\neq u_\lambda$, it is enough to show the strong convergence of $\{v_n\}$ to $v_\lambda$, i.e. $v_n\rightarrow v_\lambda$ strongly in $X_0$ as $n\rightarrow\infty$. Since $v_n\rightharpoonup v_\lambda$ weakly as $n\rightarrow\infty$ and from the embedding result $v_n\rightarrow v_\lambda$ in $L^r(\Omega)$ for $1\leq r<p_s^{*}$, hence $v_n(x)\rightarrow v_\lambda(x)$ a.e. in $\Omega$. We have the following result due to Brezis \& Lieb \cite{brezis1983relation}. As $n\rightarrow\infty$, we have
		\begin{align}\label{4e7}
		\begin{split}
		\|v_n\|&=\|v_n-v_\lambda\|+\|v_\lambda\|+o_n(1) ~~\text{and}\\
		\|v_n\|_{L^{q+1}(\Omega)}&=\|v_n-v_\lambda\|_{L^{q+1}(\Omega)}+\|v_\lambda\|_{L^{q+1}(\Omega)}+ o_n(1).
		\end{split}
		\end{align}
		Further, by the Sobolev embedding theorem, we get
		\begin{equation}
		\int_{v_n\geq \underline{u}_\lambda}|v_n^{1-\gamma}-v_\lambda^{1-\gamma}|dx=o_n(1) ~~\text{as}~~ n\rightarrow\infty
		\end{equation}
		Since $v_\lambda$ is a weak solution to the problem \eqref{main}, we get
		\begin{equation}\label{4e9}
		\|v_\lambda\|^p-\|v_\lambda\|_{L^{q+1}(\Omega)}^{q+1}-\lambda\int_\Omega v_\lambda^{1-\gamma}dx=0
		\end{equation}
		Therefore, by passing the limit $n\rightarrow\infty$ we obtain
		\begin{align}\label{4e10}
		\begin{split}
		&\int_{Q}\frac{|v_n(x)-v_n(y)|^{p-2}(v_n(x)-v_n(y))((v_n-v_\lambda)(x)-(v_n-v_\lambda)(y))}{|x-y|^{N+ps}}dxdy\\
		&=\lambda\int_{v_n\geq\underline{u}_{\lambda}}{v_n^{-\gamma}(v_n-v_\lambda)}dx +\int_\Omega v_n^q(v_n-v_\lambda)dx +o_n(1)
		\end{split}
		\end{align}
		Hence, by using \eqref{4e7}, \eqref{4e10} and \eqref{4e9} the following holds as $n\rightarrow\infty$
		\begin{equation}\label{4e11}
		\|v_n-v_\lambda\|^p=\int_\Omega |v_n-v_\lambda|^{q+1}dx + o_n(1).
		\end{equation}
		We now consider the following two cases
		\begin{itemize}
			\item[$(a)$.] $\bar{I}_\lambda(v_\lambda)\neq\bar{I}_\lambda(u_\lambda)$
			\item[$(b)$.] $\bar{I}_\lambda(v_\lambda)=\bar{I}_\lambda(u_\lambda)$
		\end{itemize}
		In case $(a)$ holds, then we are through. Otherwise, from \eqref{4e7} we get,
		\begin{equation}
		\bar{I}_\lambda(v_n-v_\lambda)=\bar{I}_\lambda(v_n)-\bar{I}_\lambda(v_\lambda)+o_n(1),  ~~\text{as}~~ n\rightarrow\infty.
		\end{equation}
		Consequently, from \eqref{4e9} we have 
		\begin{equation}\label{4e13}
		\frac{1}{p}\|v_n-v_\lambda\|^p -\frac{1}{q+1}\|v_n-v_\lambda\|_{L^{q+1}(\Omega)}^{q+1} \leq o_n(1), ~~\text{as}~~ n\rightarrow\infty.
		\end{equation}
		Therefore, from \eqref{4e11} and \eqref{4e13}, we get $\|v_n-v_\lambda\|\rightarrow0$ as $n\rightarrow\infty$. Hence $\|u_\lambda-v_\lambda\|=r$ and $v_\lambda\neq u_\lambda$. This completes the proof.
	\end{proof}
	
	\noindent Before we state the multiplicity result for Case 2, let us accumulate the necessary tools for this. Let $U(x)=(1+|x|^{p'})^{-\frac{N-sp}{p}}$ and $U_\epsilon(x)=\epsilon^{-\frac{N-sp}{p}}U(\frac{|x|}{\epsilon})$, where $\epsilon>0,~x\in\mathbb{R}^N$ and $p'=\frac{p}{p-1}$. Therefore,
	\begin{equation}\label{truncated u}
	U_\epsilon(x)=\cfrac{\epsilon^{(\frac{N-sp}{p})(\frac{p'}{p})}}{(\epsilon^{p'}+|x|^{p'})^{\frac{N-sp}{p}}}
	\end{equation}
	For a fixed $r>0$ such that,
	\begin{equation}\label{r}
	B_{4r}\subset\Omega
	\end{equation}
	let us consider, $\phi\in C_c^{\infty}(\mathbb{R}^N)$ as,
	$$\begin{cases}\label{phi}
	0\leq\phi\leq1,~~\text{in}~~\mathbb{R}^N\\
	\phi\equiv0,~~\text{in}~~\mathbb{R}^N\setminus B_{2r}\\
	\phi\equiv1~~\text{in}~~B_r.
	\end{cases}$$
	Henceforth, $r$ will denote any such number satisfying \eqref{r}. Consider the following nonnegative family of truncated functions
	\begin{equation}\label{eta}
	\eta_\epsilon(x)=U_\epsilon(x)\phi(x).
	\end{equation}
	We now prove the following proposition.
	
	\begin{proposition}
		Let $\rho>0.$ Then for every $\epsilon>0$ and for any $x\in\mathbb{R}^N\setminus B_\rho$, the following holds
		\begin{itemize}
			\item[(a)] $|\eta_\epsilon(x)|\leq C\epsilon^{\frac{(N-ps)}{p}(\frac{p'}{p})} $
			\item[(b)] $|\nabla\eta_\epsilon(x)|\leq C\epsilon^{\frac{(N-ps)}{p}(\frac{p'}{p})}$
		\end{itemize}
	\end{proposition}
	\begin{proof}
		\begin{itemize}
			\item[$(a)$] We have that, $U_\epsilon(x)=\epsilon^{-\frac{N-sp}{p}}U(\frac{|x|}{\epsilon})$. Thus, for $x\in B_\rho^c$ we have
			\begin{align}\label{eta estimate}
			|\eta_\epsilon(x)|&\leq U_\epsilon(x)\nonumber\\
			&\leq \epsilon^{-\frac{(N-sp)}{p}}\left(1+\left|\frac{\rho}{\epsilon}\right|^{p'}\right)^{-\frac{N-sp}{p}}\nonumber\\
			&\leq C\epsilon^{\frac{(N-ps)}{p}(p'-1)}\nonumber\\
			&= C\epsilon^{\frac{(N-ps)}{p}(\frac{p'}{p})}
			\end{align}
			\noindent This proves $(a)$.
			
			\item [$(b)$] For any $x\in B_\rho^c$ we have,
			\begin{align}\label{grad eta}
			|\nabla\eta_\epsilon(x)|&\leq C\epsilon^{-\frac{(N-sp)}{p}}\left[\left(1+\left|\frac{x}{\epsilon}\right|^{p'}\right)^{-\frac{N-sp}{p}}+ \frac{1}{\epsilon}\left|\frac{x}{\epsilon}\right|^{p'-1} \left(1+\left|\frac{x}{\epsilon}\right|^{p'}\right)^{-1-\frac{N-sp}{p}}\right]\nonumber\\
			& \leq C\epsilon^{-\frac{(N-sp)}{p}}\left[\left(1+\left|\frac{x}{\epsilon}\right|^{p'}\right)^{-\frac{N-sp}{p}}+ \frac{1}{\rho}\left|\frac{x}{\epsilon}\right|^{p'} \left(1+\left|\frac{x}{\epsilon}\right|^{p'}\right)^{-1-\frac{N-sp}{p}}\right]\nonumber\\
			&\leq C\epsilon^{-\frac{(N-sp)}{p}}\left(1+\frac{1}{\rho}\right) \left(1+\left|\frac{x}{\epsilon}\right|^{p'}\right)^{-\frac{N-sp}{p}}\nonumber\\
			&\leq C\epsilon^{-\frac{(N-sp)}{p}}\left(1+\frac{1}{\rho}\right) \left(\frac{\epsilon}{\rho}\right)^{\frac{(N-sp)p'}{p}}\nonumber\\
			&\leq C\epsilon^{-\frac{(N-sp)}{p}}.\epsilon^{\frac{(N-sp)p'}{p}}\nonumber\\
			&\leq C\epsilon^{\frac{(N-sp)}{p}(\frac{p'}{p})}
			\end{align}
		\end{itemize}
		Hence the proof.
	\end{proof}
	
	\begin{proposition}
		Let $r>0$ be as chosen in \eqref{r}. Then we have the following
		\begin{itemize}
			\item[(a)] For every $\epsilon>0$ and any $x\in\mathbb{R}^N$, $y\in\mathbb{R}^N\setminus B_r$ with $|x-y|\leq\frac{r}{2}$
			\begin{equation*}
			|\eta_\epsilon(x)-\eta_\epsilon(y)|\leq C\epsilon^{\frac{(N-ps)}{p}(\frac{p'}{p})}|x-y|.
			\end{equation*}
			\item[(b)] For every $\epsilon>0$ and any $x,~y\in\mathbb{R}^N\setminus B_r$,
			\begin{equation*}
			|\eta_\epsilon(x)-\eta_\epsilon(y)|\leq C\epsilon^{\frac{(N-ps)}{p}(\frac{p'}{p})}\min\{1, |x-y|\}.
			\end{equation*}
		\end{itemize}
	\end{proposition}
	\begin{proof}
		\begin{itemize}
			\item[$(a)$] For $x\in\mathbb{R}^N$, $y\in\mathbb{R}^N\setminus B_r$ with $|x-y|\leq\frac{r}{2}$, let $z$ be any point on the line segment joining $x$ and $y$, i.e. $z=tx+(1-t)y$ for some $t\in[0,1].$ Observe that
			\begin{equation}\label{z}
			|z|=|tx+(1-t)y|\geq |y|-|t(x-y)|\geq r-t\frac{r}{2}\geq \frac{r}{2}.
			\end{equation}
			Therefore, with the help of \eqref{grad eta}, \eqref{z}, we have  $|\nabla\eta_\epsilon(x)|\leq C\epsilon^{\frac{(N-ps)}{p}(\frac{p'}{p})}$ for $\rho=\frac{r}{2}.$ Hence,
			\begin{equation}\label{eta difference}
			|\eta_\epsilon(x)-\eta_\epsilon(y)|\leq C\epsilon^{\frac{(N-ps)}{p}(\frac{p'}{p})}|x-y|
			\end{equation}
			This proves the $(a).$
			\item[$(b)$] We may assume $|x-y|\geq\frac{r}{2}$, for otherwise the proof follows from part $(a)$. Therefore, from \eqref{eta estimate}, we have
			\begin{equation*}
			|\eta_\epsilon(x)-\eta_\epsilon(y)|\leq |\eta_\epsilon(x)|+|\eta_\epsilon(y)|\leq C\epsilon^{\frac{(N-ps)}{p}(\frac{p'}{p})}.
			\end{equation*}
			Thus,
			\begin{equation}\label{eta difference 1}
			|\eta_\epsilon(x)-\eta_\epsilon(y)|\leq C\epsilon^{\frac{(N-ps)}{p}(\frac{p'}{p})}\min\{1, |x-y|\}
			\end{equation}
		\end{itemize}
		This completes the proof of $(b)$.
	\end{proof}
	
	\begin{proposition}For every sufficiently small $\epsilon>0$ we have,
		\begin{equation*}
		\int_Q \cfrac{|\eta_\epsilon(x)-\eta_\epsilon(y)|^p}{|x-y|^{N+ps}}dxdy\leq2^p S^{\frac{N}{ps}}+o(\epsilon^{(N-ps)(\frac{p'}{p})})
		\end{equation*}
		where, $S$ is the best Sobolev constant.
	\end{proposition}
	\begin{proof}
		We will use the previous propositions to establish this estimate. Let $r>0$ be chosen as in \eqref{r}. Then, on using \eqref{eta}, we have
		\begin{align}\label{gagliardo split}
		\begin{split}
		\int_{\mathbb{R}^N\times\mathbb{R}^N}\cfrac{|\eta_\epsilon(x)-\eta_\epsilon(y)|^p}{|x-y|^{N+ps}}dxdy&=\int_{B_r\times B_r}\cfrac{|U_\epsilon(x)-U_\epsilon(y)|^p}{|x-y|^{N+ps}}dxdy\\
		&+2\int_{\mathbb{A}}\cfrac{|\eta_\epsilon(x)-\eta_\epsilon(y)|^p}{|x-y|^{N+ps}}dxdy\\
		&+2\int_{\mathbb{B}}\cfrac{|\eta_\epsilon(x)-\eta_\epsilon(y)|^p}{|x-y|^{N+ps}}dxdy\\
		&+\int_{B_r^c\times B_r^c}\cfrac{|\eta_\epsilon(x)-\eta_\epsilon(y)|^p}{|x-y|^{N+ps}}dxdy
		\end{split}
		\end{align}
		where, $\mathbb{A}=\{(x,y)\in\mathbb{R}^N\times\mathbb{R}^N: x\in B_r, ~y\in B_r^c~\text{and}~|x-y|>\frac{r}{2}\} ~\text{and}~ \mathbb{B}=\{(x,y)\in\mathbb{R}^N\times\mathbb{R}^N: x\in B_r, ~y\in B_r^c~\text{and}~|x-y|\leq\frac{r}{2}\}.$ We will try to estimate the last three terms of \eqref{gagliardo split}. From \eqref{eta difference 1}, we have
		\begin{align}\label{estimate on compliment}
		\begin{split}
		\int_{B_r^c\times B_r^c}\cfrac{|\eta_\epsilon(x)-\eta_\epsilon(y)|^p}{|x-y|^{N+ps}}dxdy&\leq C\epsilon^{(N-ps)(\frac{p'}{p})}\int_{B_{2r}\times\mathbb{R}^N}\cfrac{\min\{1, |x-y|^p\}}{|x-y|^{N+ps}}dxdy\\ &=o(\epsilon^{(N-ps)(\frac{p'}{p})}),~~\text{as}~\epsilon\rightarrow0.
		\end{split}
		\end{align}
		On the other hand, from \eqref{eta difference} we have,
		\begin{align}\label{esimate on B}
		\begin{split}
		\int_{\mathbb{B}}\cfrac{|\eta_\epsilon(x)-\eta_\epsilon(y)|^p}{|x-y|^{N+ps}}dxdy&\leq C\epsilon^{(N-ps)(\frac{p'}{p})}\int_{\{x\in B_r, ~y\in B_r^c~|x-y|\leq\frac{r}{2}\}}\cfrac{|x-y|^p}{|x-y|^{N+ps}}dxdy\\
		&\leq C\epsilon^{(N-ps)(\frac{p'}{p})}\int_{|x|\leq r}dx\int_{|z|\leq\frac{r}{2}}\frac{1}{|z|^{N+ps-p}}dz\\
		&=o(\epsilon^{(N-ps)(\frac{p'}{p})}),~~\text{as}~\epsilon\rightarrow0.
		\end{split}
		\end{align}
		Now, the only estimate remains to be proved is the integral over $\mathbb{A}$ in \eqref{gagliardo split}, which is the following
		\begin{equation}\label{estimate on A}
		\int_{\mathbb{A}}\cfrac{|\eta_\epsilon(x)-\eta_\epsilon(y)|^p}{|x-y|^{N+ps}}dxdy
		\end{equation}
		Since, $\eta_\epsilon(x)=U_\epsilon(x)$ in $B_r$, we have
		\begin{align}\label{p formula}
		\begin{split}
		(|\eta_\epsilon(x)-\eta_\epsilon(y)|^p)&\leq|U_\epsilon(x)-\eta_\epsilon(y)|^p\\
		&=|U_\epsilon(x)-U_\epsilon(y)+U_\epsilon(y)-\eta_\epsilon(y)|^p\\
		&\leq(|U_\epsilon(x)-U_\epsilon(y)|+|U_\epsilon(y)-\eta_\epsilon(y)|)^p\\
		&\leq2^{p-1}(|U_\epsilon(x)-U_\epsilon(y)|^p+|U_\epsilon(y)-\eta_\epsilon(y)|^p)
		\end{split}
		\end{align}
		On using \eqref{p formula} in \eqref{estimate on A}, the integral becomes
		\begin{align}\label{estimate on A1}
		\int_{\mathbb{A}}\cfrac{|\eta_\epsilon(x)-\eta_\epsilon(y)|^p}{|x-y|^{N+ps}}dxdy\leq2^{p-1}&\int_{\mathbb{A}}\cfrac{|U_\epsilon(x)-U_\epsilon(y)|^p}{|x-y|^{N+ps}}dxdy\nonumber\\
		&+2^{p-1}\int_{\mathbb{A}}\cfrac{|U_\epsilon(y)-\eta_\epsilon(y)|^p}{|x-y|^{N+ps}}dxdy
		\end{align}
		We will estimate the last term of \eqref{estimate on A1}. From \eqref{eta estimate} for $\rho=r$, when $\epsilon\rightarrow0$ we have 
		\begin{align}\label{estimate on A2}
		\begin{split}
		\int_{\mathbb{A}}\cfrac{|U_\epsilon(y)-\eta_\epsilon(y)|^p}{|x-y|^{N+ps}}dxdy&\leq\int_{\mathbb{A}}\cfrac{\left(|U_\epsilon(y)|+|\eta_\epsilon(y)|\right)^p}{|x-y|^{N+ps}}dxdy\\
		&\leq C\int_{\mathbb{A}}\cfrac{|U_\epsilon(y)|^p}{|x-y|^{N+ps}}dxdy\\
		&\leq C\epsilon^{(N-ps)(\frac{p'}{p})}\int_{\{x\in B_r, ~y\in B_r^c~|x-y|>\frac{r}{2}\}}\cfrac{|x-y|^p}{|x-y|^{N+ps}}dxdy\\
		&\leq C\epsilon^{(N-ps)(\frac{p'}{p})}\int_{|x|\leq r}dx\int_{|z|>\frac{r}{2}}\frac{1}{|z|^{N+ps-p}}dz\\
		&=o(\epsilon^{(N-ps)(\frac{p'}{p})}).
		\end{split}
		\end{align}
		Therefore, by using \eqref{estimate on compliment}, \eqref{esimate on B}, \eqref{estimate on A1} and \eqref{estimate on A2}, we have
		\begin{align*}
		\begin{split}
		\int_{\mathbb{R}^N\times\mathbb{R}^N}\cfrac{|\eta_\epsilon(x)-\eta_\epsilon(y)|^p}{|x-y|^{N+ps}}dxdy&=\int_{B_r\times B_r}\cfrac{|U_\epsilon(x)-U_\epsilon(y)|^p}{|x-y|^{N+ps}}dxdy\\
		&\hspace{0.5cm}+2^{p-1}\int_{\mathbb{A}}\cfrac{|U_\epsilon(x)-U_\epsilon(y)|^p}{|x-y|^{N+ps}}dxdy+ o(\epsilon^{(N-ps)(\frac{p'}{p})})\\
		&\leq2^p\int_{\mathbb{R}^N\times\mathbb{R}^N}\cfrac{|U_\epsilon(x)-U_\epsilon(y)|^p}{|x-y|^{N+ps}}dxdy+ o(\epsilon^{(N-ps)(\frac{p'}{p})})
		\end{split}
		\end{align*}
		For every $\epsilon>0$, the functions $U_\epsilon(x)$ are the minimizer of the problem
		\begin{align*}
		(-\Delta_p)^s u&=|u|^{p_s^*-2}u,~\text{in}~\Omega\\
		u&=0,~\text{on}~\partial\Omega
		\end{align*} and hence satisfies the following equality
		\begin{equation*}
		\int_{\mathbb{R}^N\times\mathbb{R}^N}\cfrac{|U_\epsilon(x)-U_\epsilon(y)|^p}{|x-y|^{N+ps}}dxdy=\int_{\mathbb{R}^N}|U_\epsilon(x)|^{p_s^*}dx=S^{\frac{N}{ps}}
		\end{equation*}
		Hence, we get
		\begin{equation}\label{gagliardo final}
		\int_Q \cfrac{|\eta_\epsilon(x)-\eta_\epsilon(y)|^p}{|x-y|^{N+ps}}dxdy
		\leq2^pS^{\frac{N}{ps}}+o(\epsilon^{{(N-ps)}(\frac{p'}{p})})
		\end{equation}
		This completes the proof.
	\end{proof}
	\begin{proposition}For a sufficiently small $\epsilon>0$ we have,
		\begin{itemize}
			\item[(a)] $\int_\Omega|\eta_\epsilon|^{\beta}dx\leq C\epsilon^{\left(\frac{N-ps}{p}\right)(\frac{p'}{p})\beta}$
			\item[(b)] $\int_\Omega|\eta_\epsilon|^{q+1}dx\geq C\epsilon^{N-\left(\frac{N-ps}{p}\right)(q+1)}$ 
		\end{itemize}
	\end{proposition}
	
	\begin{proof} 
		\begin{itemize}
			\item[$(a)$] From \eqref{eta estimate}, we have
			\begin{align}\label{estimate for beta}
			\int_\Omega |\eta_\epsilon(x)|^{\beta}dx\leq C\epsilon^{(\frac{N-sp}{p})(\frac{p'}{p})\beta}
			\end{align} 
			\item[$(b)$] We have
			\begin{align}\label{estimate for q plus 1}
			\int_\Omega |\eta_\epsilon(x)|^{q+1}dx &= C\int_{|x|<r}U_\epsilon^{q+1}(x)dx\nonumber\\
			&=C\int_{|x|<r}\cfrac{\epsilon^{(\frac{N-sp}{p})(\frac{p'}{p})(q+1)}}{(\epsilon^{p'}+|x|^{p'})^{\frac{N-sp}{p}(q+1)}}dx\nonumber\\
			&= C\epsilon^{(\frac{N-sp}{p})(\frac{p'}{p})(q+1)}\int_{|x|<r}\cfrac{dx}{(\epsilon^{p'} +|x|^{p'})^{\frac{N-sp}{p}(q+1)}}\nonumber\\
			&= C\epsilon^{(\frac{N-sp}{p})(\frac{p'}{p})(q+1)-(\frac{N-sp}{p})(q+1)p'}\int_{0}^{r} \cfrac{t^{N-1}}{\left(1+(\frac{t}{\epsilon})^{p'}\right)^{\frac{N-sp}{p}(q+1)}}dt\nonumber\\
			&=C\epsilon^{N-(\frac{N-sp}{p})(\frac{p'}{p})(q+1)(p-1)}\int_{0}^{\frac{r}{\epsilon}}\cfrac{y^{N-1}}{\left(1+y^{p'}\right)^{\frac{N-sp}{p}(q+1)}}dy\nonumber\\
			&\geq C\epsilon^{N-(\frac{N-sp}{p})(\frac{p'}{p})(q+1)(p-1)}\int_{1}^{\frac{r}{\epsilon}} y^{N-1-(N-ps)(q+1)}dy\nonumber\\
			&=\frac{C\epsilon^{N-(\frac{N-sp}{p})(q+1)}}{L}\left[1-\left(\frac{\epsilon}{r}\right)^L\right]\nonumber\\
			&\geq C'\epsilon^{N-(\frac{N-sp}{p})(q+1)},~\text{for some}~C'>0.
			\end{align} 
		\end{itemize}
		\noindent where, $L=-(N-(N-ps)(q+1))>0.$
		\noindent Hence the proof is complete.
	\end{proof}

	\noindent We now prove the following Lemma, when Case 2 holds.
	\begin{lemma}\label{case2}
		Suppose Case 2 holds, then for $1<p<\infty, ~p-1<q\leq p_s^{*}-1, ~0<\gamma<1$ and $\lambda\in(0,\Lambda)$, there exists a weak solution $v_\lambda$ of the problem \eqref{main} such that $v_\lambda\neq u_\lambda$. 
	\end{lemma}
	\begin{proof}
		From Theorem \ref{bounded}, we have the weak solution $u_\lambda$ of the problem \eqref{main} is bounded. Therefore, there exists positive real numbers $m$ and $M$ such that $m\leq u_\lambda(x)\leq M,\,\forall\, x\in\Omega.$ We have, by Mosconi et al. \cite{mosconi2016brezis}, that the Palais Smale(PS) condition is satisfied, if
		$$\delta_0<I_\lambda(u_\lambda)+\frac{s}{N}S^\frac{N}{ps}.$$
		{\bfseries Claim.} $ \underset{0\leq t\leq\frac{1}{2}}{\sup} I_\lambda(u_\lambda+t\eta_\epsilon)< I_\lambda(u_\lambda)+\frac{s}{N}S^\frac{N}{ps}$.\\
		By the Definitions of $I_\lambda$ and $\bar{I}_\lambda$, we have $\bar{I}_\lambda(u_\lambda+t\eta_\epsilon)=I_\lambda(u_\lambda+t\eta_\epsilon)$ and $\bar{I}_\lambda(u_\lambda)=I_\lambda(u_\lambda).$
		Using the estimates given in the {\it page 946}, of  Azorero and Alonso \cite{azorero1994some}, one can conclude that 
		\begin{equation*}
		\bar{I}_\lambda(u_\lambda+t\eta_\epsilon)\leq\bar{I}_\lambda(u_\lambda)+pt\left[\int_Q \cfrac{|u_\lambda(x)-u_\lambda(y)|^{p-2}(u_\lambda(x)-u_\lambda(y))(\eta_\epsilon(x)-\eta_\epsilon(y))}{|x-y|^{N+ps}}dxdy\right]+o(\epsilon^{\alpha})
		\end{equation*} for every $\alpha>\frac{N-ps}{p}.$
		Therefore, we have
		\begin{align}\label{estimate for case 2}
		&~\bar{I}_\lambda(u_\lambda+t\eta_\epsilon)-\bar{I}_\lambda(u_\lambda)\nonumber\\
		&=I_\lambda(u_\lambda+t\eta_\epsilon)-I_\lambda(u_\lambda)-pt\left[\int_Q\cfrac{|u_\lambda(x)-u_\lambda(y)|^{p-2}(u_\lambda(x)-u_\lambda(y))(\eta_\epsilon(x)-\eta_\epsilon(y))}{|x-y|^{N+ps}}dxdy\right.\nonumber\\ &\hspace{7cm}\left.-\lambda\int_\Omega u_\lambda^{-\gamma}\eta_\epsilon(x)dx-\int_\Omega u_\lambda^{q}\eta_\epsilon(x)dx\right]\nonumber\\
		&\leq\frac{t^p}{p}\int_Q\cfrac{|\eta_\epsilon(x)-\eta_\epsilon(y)|^{p}}{|x-y|^{N+ps}}dxdy\\
		&\hspace{1cm}+\lambda\left[t\int_\Omega u_\lambda^{-\gamma}\eta_\epsilon(x)dx+\frac{1}{1-\gamma}\int_\Omega |u_\lambda|^{1-\gamma}dx-\frac{1}{1-\gamma}\int_\Omega |u_\lambda+t\eta_\epsilon|^{1-\gamma}dx \right]\nonumber\\
		&\hspace{1cm}-\left[t\int_\Omega u_\lambda^{q}\eta_\epsilon(x)dx+\frac{1}{q+1}\int_\Omega |u_\lambda|^{q+1}dx-\frac{1}{q+1}\int_\Omega |u_\lambda+t\eta_\epsilon|^{q+1}dx \right]+o(\epsilon^{\frac{N-ps}{p}})\nonumber.
		\end{align}
		The following two inequalities holds true \cite{saoudi2017critical}. For every $a,b\geq0$ with $a\geq m$, we have
		\begin{equation}\label{estimate c1}
		\lambda\left(a^{-\gamma} b+\frac{a^{1-\gamma}}{1-\gamma}-\frac{(a+b)^{1-\gamma}}{1-\gamma}\right)\leq C_1b^\beta, ~\text{for some constant}~C_1>0.
		\end{equation} and
		\begin{equation}\label{estimate c2}
		\left(\frac{(a+b)^{q+1}}{q+1}-\frac{a^{q+1}}{q+1}-a^{q} b\right)\geq \frac{b^{q+1}}{q+1}, ~\text{for some constants}~ a,b\geq0.
		\end{equation}
		By using the above two inequalities \eqref{estimate c1} and \eqref{estimate c2}, the inequality \eqref{estimate for case 2} becomes
		\begin{align}\label{estimate of all term case 2}
		\bar{I}_\lambda(u_\lambda+t\eta_\epsilon)-\bar{I}_\lambda(u_\lambda)&\leq\frac{t^p}{p}\int_Q\cfrac{|\eta_\epsilon(x)-\eta_\epsilon(y)|^{p}}{|x-y|^{N+ps}}dxdy\nonumber\\
		&-\frac{t^{q+1}}{q+1}\int_\Omega|\eta_\epsilon|^{q+1}dx +C_1t^\beta\int_\Omega|\eta_\epsilon|^{\beta}dx
		\end{align}
		Hence, from \eqref{gagliardo final}, \eqref{estimate for beta} and \eqref{estimate for q plus 1}, we get
		\begin{align}
		\begin{split}
		I_\lambda(u_\lambda+t\eta_\epsilon)-I_\lambda(u_\lambda)&\leq\frac{(2t)^p}{p}\left(S^{\frac{N}{ps}}+o(\epsilon^{{(N-ps)}(\frac{p'}{p})})\right)\\&-C\epsilon^{N-(\frac{N-sp}{p})(q+1)}+C\epsilon^{(\frac{N-sp}{p})(\frac{p'}{p})\beta}
		\end{split}   
		\end{align}
		Therefore, we can conclude that,
		\begin{align*}
		\underset{0\leq t\leq\frac{1}{2}}{\sup}&\{ I_\lambda(u_\lambda+t\eta_\epsilon)-I_\lambda(u_\lambda)\}<\frac{s}{N}S^\frac{N}{ps}\\
		\Rightarrow\underset{0\leq t\leq\frac{1}{2}}{\sup}& I_\lambda(u_\lambda+t\eta_\epsilon)< I_\lambda(u_\lambda)+\frac{s}{N}S^\frac{N}{ps}
		\end{align*}
		Hence, by the result in \cite{mosconi2016brezis}, $\{v_n\}$ is a (PS) sequence.Thus the sequence $\{v_n\}$ has a strongly convergent subsequence, from which we conclude that $\delta_0=I_\lambda(v_\lambda)>I_\lambda(u_\lambda).$ Therefore $v_\lambda\neq u_\lambda$. This completes the proof.
	\end{proof}

	\section{Regularity of the weak solutions}\label{regularity}
	Firstly, let us recall the following elementary inequality needed for the proof of the
$L^\infty$ estimate:
	
	\begin{lemma}\label{ineq1}
		(Generalization of the Lemma 3.1 in \cite{iannizzotto2015hs}) For all $a$, $b\in\mathbb{R}$, $r\geq p$, $p\geq 2$, $k>0$ we have
		\begin{align*}
		\begin{split}
		\frac{p^p(r+1-p)}{r^p}(a|a|_k^{\frac{r}{p}-(p-1)}-b|b|_k^{\frac{r}{p}-(p-1)})^p&\leq (a|a|_k^{r-p(p-1)}-b|b|_k^{r-p(p-1)})(a-b)^{p-1}
		\end{split}
		\end{align*}
		assuming $a \geq b$.
	\end{lemma}
	\begin{proof}
		Define 
		\[h(t)=\begin{cases}
		\text{sgn}(t)|t|^{\frac{r}{p}-1}, & |t|<k \\
		\frac{p}{r}\text{sgn}(t)k^{\frac{r}{p}-1}, & |t|\geq k.              
		\end{cases}\]
		Observe that 
		\begin{align*}
		\begin{split}
		\int_{b}^{a}h(t)dt&=\frac{p}{r}(a|a|_k^{\frac{r}{p}-(p-1)}-b|b|_k^{\frac{r}{p}-(p-1)}).
		\end{split}
		\end{align*}
		Similarly,
		\begin{align*}
		\begin{split}
		\int_{b}^{a}h(t)^pdt&\leq\frac{1}{r+1-p}(a|a|_k^{r-p(p-1)}-b|b|_k^{r-p(p-1)}).
		\end{split}
		\end{align*}
		On using the Cauchy-Schwartz inequality we obtain
		\begin{align*}
		\begin{split}
		\left(\int_{b}^{a}h(t)dt\right)^p&\leq(a-b)^{p-1}\int_{b}^{a}h(t)^pdt.
		\end{split}
		\end{align*}
		Thus 
		\begin{align*}
		\begin{split}
		\frac{p^p}{r^p}(a|a|_k^{\frac{r}{p}-(p-1)}-b|b|_k^{\frac{r}{p}-(p-1)})^p&=\left(\int_{b}^{a}h(t)dt\right)^p\\
		&\leq(a-b)^{p-1}\int_{b}^{a}h(t)^pdt\\
		&\leq\frac{(a-b)^{p-1}}{r+1-p}(a|a|_k^{r-p(p-1)}-b|b|_k^{r-p(p-1)}).
		\end{split}
		\end{align*}
	\end{proof}
	\begin{lemma}
		The bound of the weak solution $u_0$ to ($P$) by the first eigen function.
	\end{lemma}
	\begin{proof}
		The proof is again by contradiction., i.e. $\forall\eta>0$ let $|\Omega_{\eta}|=|supp\{(\eta\phi_1-u_0)^{+}\}|>0$. Define $v_{\eta}=(\eta\phi_1-u_0)^{+}$. For $0<t<1$ define $\xi(t)=I(v_{\epsilon}+v_{\eta})$. Thus 
		\begin{align}
		\begin{split}
		\xi'(t)&=\langle I'(u_0+tv_{\eta}),v_{\eta}\rangle\\
		&=\langle(-\Delta_p)^s(u_0+tv_{\eta})-(u_0+tv_{\eta})^{-\alpha}-f(x,u_0+tv_{\eta}),v_{\eta}\rangle.
		\end{split}
		\end{align}
		Similarly,
		\begin{align}
		\begin{split}
		\xi'(1)&=\langle I'(u_0+v_{\eta}),v_{\eta}\rangle\\
		&=\langle I'(\eta\phi_1),v_{\eta}\rangle\\
		&=\langle(-\Delta_p)^s(\eta\phi_1)-(\eta\phi_1)^{-\alpha}-f(x,\eta\phi_1),v_{\eta}\rangle<0
		\end{split}
		\end{align}
		for sufficiently small $\eta>0$. Moreover,
		\begin{align}
		\begin{split}
		-\xi'(1)+\xi'(t)&=\langle(-\Delta_p)^s(u_0+tv_{\eta})-(-\Delta_p)^s(u_0+v_{\eta})\\
		&+((u_0+v_{\eta})^{-\alpha}-(u_0+tv_{\eta})^{-\alpha})+(f(x,u_0+v_{\eta})-f(x,u_0+tv_{\eta})),v_{\eta}\rangle.
		\end{split}
		\end{align}
		Since $s^{-\alpha}+f(x,s)$ is a uniformly nonincreasing function with respect to $x\in\Omega$ for sufficiently small $s>0$. Also from the monotonicity of $(-\Delta_p)^s$ we have, for sufficiently small $\eta>0$, $0\leq \xi'(1)-\xi'(t)$. From the Taylor series expansion and the fact that $K(v_{\epsilon})<\epsilon$ we have $\exists 0<\theta<1$ such that 
		\begin{align}
		\begin{split}
		0&\leq I(u_0+v_{\eta})-I(u_0)\\
		&=\langle I'(u_0+\theta v_{\eta}),v_{\eta}\rangle\\
		&=\xi'(\theta).
		\end{split}
		\end{align}  
		Thus for $t=\theta$ we have $\xi'(\theta)\geq 0$ which is a contradiction to $\xi'(\theta)\leq\xi'(1)<0$ as obtained above. Thus $u_0\geq \eta\phi_1$ for some $\eta>0$.\\
		It has already been proved in lemma \ref{bounded} that $u_0$ is $L^{\infty}$ bounded. Therefore, there exists $B>0$, sufficiently large, such that $0\leq u_0\leq B \phi_1$.
	\end{proof}
	\begin{theorem}\label{bounded}
		Let $f:\Omega\times\mathbb{R}\rightarrow\mathbb{R}$ be a nonlinear, Carath\'{e}odory function which  satisfy the growth condition $|f(x,t)|\leq a(1+|t|^{q-1})$ a.e. in $\Omega$ and for all $t\in\mathbb{R}$ ($a>0$, $1\leq q\leq p_s^*$), then for any weak solution $u\in X$ we have $u\in L^{\infty}(\Omega)$.
	\end{theorem}
	\begin{proof}
		{\it Case 1}:~({\bf Subcritical case $1\leq q<p_s^*$})~
		Let $u$ be a weak solution to the given probem and $\gamma=\left(\frac{p_s^*}{p}\right)^{\frac{1}{p}}$. For every $r\geq p(p-1)$, $p \geq 2$, $k>0$, the mapping $t\mapsto t|t|_k^{r-p}$ is Lipshitz in $\mathbb{R}$. Thus $u|u|_k^{r-p}\in X$. Here, in general for any $t$ in $\mathbb{R}$ and $k>0$, we have defined $t_k=\text{sgn}(t)\min\{|t|,k\}$.  We apply the Sobolev inequality, previous lemma, test with the test function $u|u|_k^{r-p}$ and on using the growth condition of $f$ which is given in the theorem to get
		\begin{align}\label{crit_subcrit}
		\begin{split}
		\|u|u|_k^{\frac{r}{p}-1}\|_{p_s^*}^p&\leq C \|u|u|_k^{\frac{r}{p}-1}\|_{X}^p\\
		&\leq C\frac{r^p}{r+1-p}\langle u,u|u|_k^{r-p} \rangle_{X}\\
		&\leq Cr^p\int_{\Omega}|f(x,u)||u||u|_k^{r-p}dx\\
		&\leq Cr^{p}\int_{\Omega}(|u||u|_k^{r-p}+|u|^q|u|_k^{r-p}+|u|^{1-\alpha}|u|_k^{r-p})dx
		\end{split}
		\end{align}
		for some $C>0$ independent of $r\geq p$ and $k>0$. On applying the Fatou's lemma as $k\rightarrow\infty$ gives 
		\begin{align}\label{eq0}
		\begin{split}
		\|u\|_{\gamma^p r}&\leq Cr^{\frac{p}{r}}\left\{\int_{\Omega}(|u|^{r-(p-1)}+|u|^{r+q-p}+|u|^{r-p-\alpha+1})dx\right\}^{1/r}.
		\end{split}
		\end{align}
		Here onwards we try to develope an argument to guarantee that $u\in L^{p_1}(\Omega)$ for all $p_1\geq 1$. Towards this, we divide our attempt into two cases, viz. sub and supercritical cases. Define a recurssive sequence $(r_n)$ by choosing $\mu>\mu_0$ and setting $r_0=\mu$, $r_{n+1}=\gamma^p r_n+p-q$. \\
		{\it Case 1:}~(Subcritical case $q<p_s^*$)\\
		We fix $\mu=p_s^*+p-q>\max\{p,\mu_0\}$. Since, $r_0+q-p=p_s^*$ we have $u\in L^{r_0+q-p}(\Omega)$ (because $u$ is a weak solution and by the embedding result). Therefore, on choosing $r=r_0$ in (\ref{eq0}) we obtain a finite right side, so $u\in L^{\gamma^p r_0}(\Omega)=L^{r_1+q-p}(\Omega)$. Iterating this argument and using the fact $r\mapsto r^{1/r}$ is bounded in $[2,\infty)$ for all $n$, we have $u\in L^{\gamma^p r_n}(\Omega)$. We further have 
		\begin{align}\label{eq1}
		\begin{split}
		\|u\|_{\gamma^p r_n}&\leq H(n,\|u\|_{p_s^*}).
		\end{split}
		\end{align}
		Arguments from Iannizzotto guarantees
		\begin{align}\label{estimate1}
		\begin{split}
		\|u\|_{p_1}&\leq H(p_1,\|u\|_{p_s^*}), p_1\geq 1.
		\end{split}
		\end{align}
		We now attempt to improve the estimate in (\ref{estimate1}) by making the function $H$ free of $p_1$. Set $\gamma'=\frac{\gamma}{\gamma-1}$. Thus by (\ref{estimate1}) and H\"{o}lder's inequality we have
		\begin{align*}
		\begin{split}
		\||u|+|u|^{q}\|_{\gamma'}&\leq H(\|u\|_{p_s^*})
		\end{split}
		\end{align*} 
		Therefore for $r\geq p(p-1)$ we have 
		\begin{align*}\label{estimate2}
		\begin{split}
		\||u|^{r-(p-1)}+|u|^{r+q-p}+|u|^{r-p-\alpha+1}\|_{\gamma'}&\leq \||u|+|u|^{q}+|u|^{1-\alpha}\|_{\gamma'}\||u|^{r-p}\|_{\gamma}\\
		&\leq H(\|u\|_{p_s^*})\|u\|_{\gamma(r-p)}^{r-p}\\
		&\leq H(\|u\|_{p_s^*})\|u\|_{\gamma^{p-1}(r-p)}^{r-p}\\
		&\leq H(\|u\|_{p_s^*})|\Omega|^{\frac{1}{\gamma^{p-1} r}}\|u\|_{\gamma^{p-1} r}^{r-p}
		\end{split}
		\end{align*}
		We note that $t\mapsto |\Omega|^{p/(\gamma^{p-1} t)}$ is a bounded map in $[p,\infty)$ and hence
		\begin{align}
		\begin{split}
		\||u|^{r-(p-1)}+|u|^{r+q-p}+|u|^{r-p-\alpha+1}\|_{\gamma'}&\leq M(\|u\|_{p_s^*})\|u\|_{\gamma^{p-1} r}^{r-p}
		\end{split}
		\end{align}
		\label{estimate3}
		For a sufficiently large $n$ we define $r=\gamma^{n-1}>>p$ and further set $v=\frac{u}{H(\|u\|_{p_s^*})^{1/p}}$. Using these choices in (\ref{eq0}) and the recurssive formula we obtain we get 
		\begin{align}
		\|u\|_{\gamma^{n}+p-1}^{\gamma^{n-1}}&\leq H(\|u\|_{p_s^*})\|u\|_{\gamma^{n-p+2}}^{\gamma^{n-1}-p}.
		\end{align}
		On using the definition of $v$ and iterating we get 
		\begin{align*}
		\begin{split}
		\|v\|_{\gamma^{n+p-1}}&\leq \|v\|_{\gamma^{n+p-2}}^{1-p\gamma^{1-n}}\\
		&\leq \|v\|_{\gamma^{n+p-3}}^{(1-p\gamma^{1-n})(1-(p-1)\gamma^{2-n})}\\ 
		&.....\\
		&\leq  \|v\|_{\gamma^{p}}^{\prod_{i=1}^{n-1}[1-p\gamma^{i-n}]}    
		\end{split}
		\end{align*}
		It is easy to see that the product $\prod_{i=1}^{n-1}[1-p\gamma^{i-n}]$ is bounded in $\mathbb{R}$ and hence for all $n$ we have 
		\begin{align*}
		\begin{split}
		\|v\|_{\gamma^{n+p-1}}&\leq  \|v\|_{\gamma^{p}}^{\prod_{i=1}^{n-1}[1-p\gamma^{i-n}]}<\infty.
		\end{split}
		\end{align*}
		Reverting back to $u$ and recalling the fact that $\gamma^{n-1}\rightarrow\infty$ as $n\rightarrow\infty$, we find that there exists $M\in C(\mathbb{R}^+)$ such that $\|u\|_{p_1}\leq M(\|u\|_{p_s^*})$ for all $p_1 \geq 1$. The function $M$ here has been obtained from the function $H$ which was previously. Therefore, we have $\|u\|_{\infty}\leq \infty$.\\
		{\it Case 2}:~(Critical case $q=p_s^*$)~We use $r=q+p-1>p$ in (\ref{crit_subcrit}) and fix $\delta>0$ such that $Cr^p\delta<\frac{1}{p}$. There exists $K_0>0$ (depending on $u$) such that $(\int_{\{|u|>K_0\}}|u|^q)^{1-\frac{p}{q}}\leq \delta$. Thus by the H\"{o}lder's inequality in $(\int_{\{|u|>K_0\}}|u|^q)^{1-\frac{p}{q}}\leq \delta$ we get 
		\begin{align}
		\begin{split}
		\int_{\Omega}|u|^q|u|_k^{r-p}dx&\leq K_0^{q+r-p}|\{|u|\leq K_0\}|+\int_{\{|u|>K_0\}}|u|^q|u|_k^{r-p}dx\\
		&\leq K_0^{q+r-p}|\{|u|\leq K_0\}|\\
		&+\left(\int_{\Omega}(|u|^p|u|_k^{r-p})^{q/p}dx\right)^{\frac{p}{q}}\left(\int_{\{|u|>K_0\}}(|u|^q|dx\right)^{1-\frac{p}{q}}\\
		&\leq K_0^{q+r-p}|\Omega|+\delta\|u|u|_k^{\frac{r-p}{p}}\|_q^p.
		\end{split}
		\end{align}
		Using the choice of $Cr\delta<\frac{1}{2}$ and the Lemma \ref{ineq1} to obtain
		\begin{align}
		\begin{split}
		\frac{1}{2}\|u|u|_k^{\frac{r-p}{p}}\|_q^p&\leq C(q+p-1)^{p}(\|u\|_q^q+K_0^{q+p-1}|\Omega|+\|u\|_{q-\alpha}^{q-\alpha}).
		\end{split}
		\end{align}
		On passing the limit $k\rightarrow\infty$ we have
		\begin{align}
		\begin{split}
		\|u\|_{\frac{q(q+p-1)}{p}}&\leq N(K_0,\|u\|_q).
		\end{split}
		\end{align}
		The rest of the proof follows from the same argument and words as in the subcritical case.
	\end{proof}
	\begin{lemma}
	~$|(-\Delta_p)^su|\leq K$ in $B_r(x)$ where $u$ is a weak solution to the problem ($P$).
	\end{lemma}
	\begin{proof}~By definition 
	\begin{align}
	\begin{split}
	|(-\Delta_p)^su(x)|&=C\int_{B_r(x)}\frac{|u(x)-u(y)|^{p-1}}{|x-y|^{N+ps}}dy\\
	&=C\int_{B_r(x)}\frac{|u(x)-u(y)|^{p-1}}{|x-y|^{N+(p-1)s}}\frac{1}{|x-y|^{ps-(p-1)s}}dy.
	\end{split}
	\end{align}
	Converting this into polar coordinates and applying the H\"{o}lder's inequality we obtain $|(-\Delta_p)^su(x)|\leq C$ in $B_r(x)$.
	The proof follows {\it verbatim} when $|(-\Delta_p)^su'|$ is considered.
	\end{proof}
	We now generalize two of the results of Iannizzoto.
	\begin{lemma}\label{calpha1}
		There exists $0<\alpha\leq s$ such that $[u/\delta^s]_{C^{\alpha}(\overline{\Omega})}\leq C$ for all weak solutions of the problem.
	\end{lemma}
	\begin{proof}
		Observe that 
		\begin{align}\label{hold1}
		\begin{split}
		\text{Tail}(u/\delta^s;x,R_0)^{p-1}&=R_0^{ps}\int_{B_{R_0}(x)^c}\frac{|u(y)|^{p-1}}{\delta(y)^{s(p-1)}|x-y|^{N+sp}}dy\\
		&\leq CR_0^{ps}\left(\int_{B_{2R_0}(x_0)\setminus B_{R_0(x_0)}}\frac{\|u/\delta^s\|_{L^{\infty}(B_{2R_0}(x_0))}^{p-1}}{|x-y|^{N+ps}}dy\right.\\
		&\left.+\int_{B_{2R_0(x_0)^c}}\frac{|u(y)|^{p-1}}{\delta(y)^{s(p-1)}|x-y|^{N+ps}}dy\right)\\
		&\leq C\left(\|u/\delta^s\|_{L^{\infty}(B_{2R_0}(x_0))}^{p-1}+R_0^{ps}\int_{B_{2R_0(x_0)^c}}\frac{|u(y)|^{p-1}}{\delta(y)^{s(p-1)}|x-y|^{N+ps}}dy\right)\\
		&=CQ(u/\delta^s;x_0,2R_0).
		\end{split}
		\end{align}
		Here $K$ is the bound of $(-\Delta_p)^su$ in $B_{2R_0}(x_0)$ (refer Appendix). Thus we obtained $Q(u/\delta^s;x_0,R_0)\leq CQ(u/\delta^s;x_0,2R_0)$ which implies the following H\"{o}lder seminorm estimate.
		\begin{align}\label{estimate1}
		\begin{split}
		[u/\delta^s]_{C^{\alpha}(B_{R_0}(x_0))}&\leq C[(KR_0^{ps})^{1/(p-1)}+Q(u/\delta^s;x_0,2R_0)]R_0^{-\alpha}.
		\end{split}
		\end{align}
		We assume that $\alpha\in(0,s]$.
		Let $\Omega'\Subset\Omega$. Then we have through the compactness of $\Omega'$ and the the estimate (\ref{estimate1}) that $\|u/\delta^s\|_{C^{\alpha}(\overline{\Omega'})}\leq C$.\\
		Let $\Pi:V\rightarrow \partial\Omega$ be a metric projection map defined as $\Pi(x)=\text{Argmin}_{y\in\partial\Omega^c}\{|x-y|\}$ where $V=\{x\in\overline{\Omega}:d(x,\partial\Omega)\leq \rho\}$. By (\ref{estimate1}) we have 
		\begin{align}\label{estimate2}
		\begin{split}
		[u/\delta^s]_{C^{\alpha}(B_{r/2}(x))}&\leq C[(Kr^{ps})^{1/(p-1)}+\|u/\delta^s\|_{L^{\infty}(B_r(x))}+\text{Tail}(u/\delta^s;x,r)]r^{-\alpha}.
		\end{split}
		\end{align}
		We now try to control the growth of the terms on the right hand side of (\ref{estimate2}). The first term is trivially controlled since $\alpha\leq s\leq \frac{sp}{p-1}$. The other terms are controlled uniformly due to the compactness of the set $V$.
	\end{proof}
	\begin{lemma}\label{calpha2}
		There exists $0<\alpha\leq s$ such that $[u']_{C^{\alpha}(\overline{\Omega})}\leq C$ for all weak solutions of the problem.
	\end{lemma}
	\begin{proof}
		Observe that 
		\begin{align}\label{hold1}
		\begin{split}
		\text{Tail}(u';x,R_0)^{p-1}&=R_0^{ps}\int_{B_{R_0}(x)^c}\frac{|u'(y)|^{p-1}}{|x-y|^{N+sp}}dy\\
		&\leq CR_0^{ps}\left(\int_{B_{2R_0}(x_0)\setminus B_{R_0(x_0)}}\frac{\|u'\|_{L^{\infty}(B_{2R_0}(x_0))}^{p-1}}{|x-y|^{N+ps}}dy\right.\\
		&\left.+\int_{B_{2R_0(x_0)^c}}\frac{|u'(y)|^{p-1}}{|x-y|^{N+ps}}dy\right)\\
		&\leq C\left(\|u'\|_{L^{\infty}(B_{2R_0}(x_0))}^{p-1}+R_0^{ps}\int_{B_{2R_0(x_0)^c}}\frac{|u'(y)|^{p-1}}{|x-y|^{N+ps}}dy\right)\\
		&=CQ(u';x_0,2R_0).
		\end{split}
		\end{align}
		Here $K$ is the bound of $(-\Delta_p)^su'$ in $B_{2R_0}(x_0)$ (refer Appendix). Thus we obtained $Q(u';x_0,R_0)\leq CQ(u';x_0,2R_0)$ which implies the following H\"{o}lder seminorm estimate.
		\begin{align}\label{estimate1}
		\begin{split}
		[u']_{C^{\alpha}(B_{R_0}(x_0))}&\leq C[(KR_0^{ps})^{1/(p-1)}+Q(u';x_0,2R_0)]R_0^{-\alpha}.
		\end{split}
		\end{align}
		We assume that $\alpha\in(0,s]$.
		Let $\Omega'\Subset\Omega$. Then we have through the compactness of $\Omega'$ and the the estimate (\ref{estimate1}) that $\|u'\|_{C^{\alpha}(\overline{\Omega'})}\leq C$.\\
		Let $\Pi:V\rightarrow \partial\Omega$ be a metric projection map defined as $\Pi(x)=\text{Argmin}_{y\in\partial\Omega^c}\{|x-y|\}$ where $V=\{x\in\overline{\Omega}:d(x,\partial\Omega)\leq \rho\}$. By (\ref{estimate1}) we have 
		\begin{align}\label{estimate2}
		\begin{split}
		[u']_{C^{\alpha}(B_{r/2}(x))}&\leq C[(Kr^{ps})^{1/(p-1)}+\|u'\|_{L^{\infty}(B_r(x))}+\text{Tail}(u';x,r)]r^{-\alpha}.
		\end{split}
		\end{align}
		We now try to control the growth of the terms on the right hand side of (\ref{estimate2}). The first term is trivially controlled since $\alpha\leq s\leq \frac{sp}{p-1}$. The other terms are controlled uniformly due to the compactness of the set $V$.
	\end{proof}

	\section{Appendix}
	
	\noindent We now prove the following two Lemmas with the help of Lemma \ref{gateaux estimate refer}, to establish the G\^{a}teaux differentiability of the functional $I_\lambda\colon X_0\rightarrow\mathbb{R}$, for $0<\gamma<1$.
	\begin{lemma}\label{gateaux estimate refer}
		For every $0<\gamma<1$, there exists $C_\gamma>0$, depending on $\gamma$, such that the following inequality holds true
		\begin{equation}\label{apend 3}
		\int_{0}^{1}|a+tb|^{-\gamma}dt\leq C_\gamma\left(\underset{t\in[0, 1]}{\max}~|a+tb| \right)^{-\gamma}
		\end{equation}
	\end{lemma}
	\begin{proof}
		The proof of this can be found in Lemma A.1. of \cite{takavc2002fredholm}.	
	\end{proof}
	\begin{lemma}\label{differentiability}
		Let $0<\gamma<1,~1<p<\infty, ~p-1<q\leq p_s^{*}-1$ and $\phi_1$ be the first eigenvector of the fractional $p$-Laplacian operator. Suppose $u, v\in X_0$ with $u\geq\epsilon\phi_1$, for some $\epsilon>0$. Then we have
		\begin{align}\label{5.17}
		\begin{split}
		\underset{t\rightarrow0}{\lim}\frac{I_\lambda(u+tv)-I_\lambda(u)}{t}=\int_Q &\cfrac{|u(x)-u(y)|^{p-2}(u(x)-u(y))}{|x-y|^{N+ps}}(v(x)-v(y))dxdy\\&-\lambda\int_\Omega u^{-\gamma}vdx-\int_\Omega u^qvdx
		\end{split}
		\end{align}
	\end{lemma}
	\begin{proof}
		In order to estimate \eqref{5.17}, it is enough to prove the convergence of the singular term $\int_\Omega u^{-\gamma}vdx$. Let $v\in X_0$ and $t>0$ be sufficiently small. Then we have
		\begin{align}
		\begin{split}
		0\leq\frac{I_\lambda(u+tv)-I_\lambda(u)}{t}=&\frac{1}{p}\left(\frac{\|u+tv\|^p-\|u\|^p}{t}\right)-\lambda\left(\frac{F(u+tv)-F(u)}{t}\right)\\
		&-\frac{1}{q+1}\int_\Omega\left(\frac{|u+tv|^{q+1}-|u|^{q+1}}{t}\right)
		\end{split}
		\end{align}
		where, $$F(u)=\frac{1}{1-\gamma}\int_\Omega(u^{+})^{1-\gamma}dx, ~~\text{for all}~x\in X_0.$$
		We see that as $t\rightarrow0^+$, we get
		\begin{itemize}
			\item[$(a)$] $\frac{\|u+tv\|^p-\|u\|^p}{t}\longrightarrow p\int_Q \cfrac{|u(x)-u(y)|^{p-2}(u(x)-u(y))}{|x-y|^{N+ps}}(v(x)-v(y))dxdy$
			\item[$(b)$] $\frac{1}{q+1}\int_\Omega\left(\frac{|u+tv|^{q+1}-|u|^{q+1}}{t}\right)dx\longrightarrow\int_\Omega |u|^qvdx$
		\end{itemize}
		We now define for $z\in\mathbb{R}\setminus\{0\}$,
		\begin{align}
		V(x)&=\frac{1}{1-\gamma}\frac{d}{dz}(z^+)^{1-\gamma}\nonumber\\&=
		\begin{cases}
		z^{-\gamma},~\text{if}~z>0\\ 0,~\text{if}~z<0
		\end{cases}
		\end{align}
		Therefore, for every $x\in\Omega$
		\begin{equation}\label{dominated term}
		\frac{F(u+tv)-F(u)}{t}=\int_\Omega\left(\int_{0}^{1}V(u+stv)ds\right)vdx
		\end{equation}
		Hence we get for all $x\in\Omega$, $u(x)>0$, and
		$$\underset{t\rightarrow0^+}{\lim}\int_{0}^{1}V(u(x)+stv(x)ds=V(u(x))=u(x)^{-\gamma}$$
		Also we have
		$$\left|\int_{0}^{1}V(u(x)+stv(x)ds\right|\leq\int_{0}^{1}\left|u(x)+stv(x)\right|ds$$
		Now using the estimate in the previous Lemma \ref{apend 3}, we get
		\begin{align*}
		\left|\int_{0}^{1}V(u(x)+stv(x)ds\right|&\leq C_\gamma\left(\underset{s\in[0, 1]}{\max}~|u(x)+stv(x)| \right)^{-\gamma}\\
		&\leq C_\gamma u(x)^{-\gamma}\\
		&\leq C_\gamma (\epsilon\phi_1(x))^{-\gamma}\\
		&=C_{\epsilon, \gamma}\phi_1(x)^{-\gamma}
		\end{align*}
		where, the constant $C_{\epsilon, \gamma}>0$ is independent of $x\in\Omega.$ Therefore, by the Hardy's inequality and for all $v\in X_0$, we have $v\phi_1^{-\gamma}\in L^1(\Omega).$ Hence the Lemma follows by applying Lesbegue dominated convergence theorem in \eqref{dominated term} and taking the limit as $t\rightarrow0^+$. In fact we have the following Corollary to the Lemma \ref{differentiability}.	
	\end{proof}
	\begin{corollary}
		Let $0<\gamma<1,~1<p<\infty, ~p-1<q\leq p_s^{*}-1.$ If $u\in X_0$ is such that $u\geq\epsilon\phi_1$, for some $\epsilon>0$. Then the functional $I_\lambda\colon X_0\rightarrow\mathbb{R}$ is G\^{a}teaux differentiable at $u$. The G\^{a}teaux derivative $I_\lambda(u)$ at $u$ is given by
		\begin{align}
		\begin{split}
		\langle I_\lambda(u), v\rangle=\int_Q &\cfrac{|u(x)-u(y)|^{p-2}(u(x)-u(y))}{|x-y|^{N+ps}}(v(x)-v(y))dxdy\\
		&-\lambda\int_\Omega u^{-\gamma}vdx-\int_\Omega u^qvdx,~\text{for all}~v\in X_0.
		\end{split}
		\end{align}
	\end{corollary}
	
	\begin{lemma}\label{C1}
		Let $0<\gamma<1,~1<p<\infty, ~p-1<q\leq p_s^{*}-1.$ Let $w\in X_0$ is such that $w\geq\epsilon\phi_1$, for some $\epsilon>0$. For each $x\in\Omega$, we consider
		$$f_\lambda(x, s)=\begin{cases}
		\lambda w(x)^{-\gamma}+w(x)^q,~~\text{if}~s<w(x)\\
		\lambda s^{-\gamma}+s^q,~~\text{if}~s\geq w(x)
		\end{cases}$$ with $F_\lambda(x, s)=\int_{0}^{s}f_\lambda(x, t)dt.$ For each $u\in X_0$ we define 
		$$\bar{I}_\lambda(u)=\frac{1}{p}\int_Q \cfrac{|u(x)-u(y)|^{p}}{|x-y|^{N+ps}}dxdy-\int_\Omega F_\lambda(x, u)dx.$$ Then the energy functional $\bar{I}_\lambda$ belongs to $C^1(X_0, \mathbb{R}).$
	\end{lemma}
	\begin{proof}
		To establish the result we emphasize only on the singular term. Let 
		$$g(x,s)=\begin{cases}
		w(x)^{-\gamma},~\text{if}~s<w(x)\\
		s^{-\gamma},~\text{if}~s\geq w(x)
		\end{cases}$$ where, $w\in X_0$ such that $w\geq\epsilon\phi_1.$ Let us define $G(x,s)=\int_{0}^{s}g(x,t)dt$ and $J(u)=\int_\Omega G(x,u)dx$. Proceeding with the arguments as in Lemma \ref{differentiability}, we get $J(u)$ has a G\^{a}teaux derivative $J'(u)$ for all $u\in X_0$ and it is given by
		$$\langle J'(u), v\rangle=\int_\Omega\left(\max\{u(x), w(x)\}\right)^{-\gamma}v(x)dx.$$
		Now, let $u_n\in X_0$ be such that $u_n\rightarrow u$. Then we have, for all $v\in X_0$
		\begin{align*}
		\left|\langle J'(u_n)-J'(u), v\rangle\right|&=\left|\int_\Omega\left[\left(\max\{u_n(x), w(x)\}\right)^{-\gamma}-\left(\max\{u(x), w(x)\}\right)^{-\gamma}\right]v(x)dx\right|\\
		&\leq2\int_\Omega w^{-\gamma}|v|dx\\
		&\leq2\epsilon^{-\gamma}\int_\Omega\phi_1^{-\gamma}|v|dx.
		\end{align*}
		Now as in Lemma \ref{differentiability}, using the Hardy's inequality we conclude that $\phi_1^{-\gamma}v\in L^1(\Omega)$. Hence by Lesbegue dominated convergence theorem we conclude that the G\^{a}teaux derivative of $J$ is continuous which guaranties that $J\in C^1(X_0, \mathbb{R})$.
	\end{proof}
	
	\section*{Acknowledgement}
	The author S. Ghosh, thanks the Council of Scientific and Industrial Research (C.S.I.R), India, for the financial assistantship received to carry out this research work. Both the authors thanks the research facilities received from the Department of Mathematics, National Institute of Technology Rourkela, India.
	
	\bibliographystyle{plain}

\end{document}